\documentclass[reqno,11pt,twoside]{article}
\usepackage[hmargin=1.25in, vmargin=1.25in, marginparwidth=1in, marginparsep=0.1in, a4paper, centering]{geometry}
\usepackage{amsmath, amsthm, amssymb,amsfonts}
\usepackage{fourier}
\usepackage[lining]{ebgaramond}
\usepackage[font=small, labelfont=bf, labelsep=colon]{caption}
\usepackage{graphicx}
\usepackage{xcolor}%
\usepackage{subfig}
\usepackage{enumerate}

\usepackage[colorlinks,allcolors=blue,pagebackref]{hyperref}
\usepackage{pifont}
\renewcommand*{\backrefalt}[4]{%
\ifcase #1 %
No citations%
\or
\ding{43}~p.~#2%
\else
\ding{43}~pp.~#2%
\fi}

\theoremstyle{plain}
\newtheorem{theorem}{Theorem}[section]

\newtheorem{innercustomthm}{Theorem}
\newenvironment{customthm}[1]
  {\renewcommand\theinnercustomthm{#1}\innercustomthm}
  {\endinnercustomthm}

\newtheorem{proposition}[theorem]{Proposition}%
\newtheorem{corollary}[theorem]{Corollary}
\newtheorem{lemma}[theorem]{Lemma}
\newtheorem{korollar}[theorem]{Corollary}

\theoremstyle{remark}
\newtheorem{remark}[theorem]{Remark}%

\theoremstyle{definition}
\newtheorem{example}[theorem]{Example}%

\numberwithin{equation}{section}
\newcommand{\R}{\mathbb{R}}

\renewcommand\footnotemark{}

\title{Behavior of Absorbing and Generating $p$-Robin Eigenvalues in Bounded and Exterior Domains} 

\author{Lukas Bundrock\thanks{\textbf{LB}: The University of Alabama, USA, lbundrock@ua.edu}
\and
Tiziana Giorgi\thanks{\textbf{TG}: The University of Alabama, USA, tgiorgi@ua.edu}
\and
Robert Smits
\thanks{\textbf{RS}: The University of Alabama, USA, rgsmits@ua.edu}
}

\begin{document}
\maketitle

\begin{abstract}
We establish rigorous quantitative inequalities for the first eigenvalue of the generalized $p$-Robin problem, for both the classical diffusion absorption case, where the Robin boundary parameter $\alpha$ is positive, and the superconducting {generation regime} (\(\alpha<0\)), where the boundary acts as a source. In bounded domains, we use a unified approach to derive a precise asymptotic behavior for all $p$ and all small real $\alpha$, improving existing results in various directions, including requiring weaker boundary regularity for the case of the classical 2-Robin problem, studied in the fundamental work by Ren{\'e} Sperb. In exterior domains, we characterize the existence of eigenvalues, establish general inequalities and asymptotics as \(\alpha\to 0\) for the first eigenvalue of the exterior of a ball, and obtain some sharp geometric inequalities for convex domains in two dimensions.
\end{abstract}









\tableofcontents

\section{Introduction}
The generalized Robin problem for the $p$-Laplace operator, $p$-Robin problem, along with its principal eigenvalue, emerges as a tool in diverse physical contexts involving both diffusive-absorptive and source phenomena. This eigenvalue, when it exists, plays a role in numerous and significant applications. For example, in bounded domains, the principal eigenvalue is pivotal in determining the onset temperature of superconductivity in zero field conditions (see \cite{giorgi2007eigenvalue} and \cite{kachmar2025laplace}). Here, the domain boundary acts as a source, and the governing operator coincides precisely with the Laplacian—a fundamental mathematical operator central to modeling heat transfer, fluid dynamics, and wave propagation. The $p$-Laplacian is also used to describe the behavior of complex fluids. Specifically, shear-thickening non-Newtonian fluids are captured by the $p$-Laplacian with $p \geq 2$, while shear-thinning fluids require the $p$-Laplacian with $p \in (1,2)$, \cite{astarita1974principles}. Thus, understanding the $p$-Robin problem and optimizing its eigenvalue properties has implications across both theoretical and applied sciences, motivating further research into shape optimization and boundary condition effects.

In this paper, we consider the $p$-Robin problem in bounded and exterior domains and derive results that fill in some existing gaps in the behavior of the principal eigenvalue. For a Lipschitz domain $\mathcal{U} \subset \mathbb{R}^n$ with $n \geq 2$, $\alpha \in \mathbb{R}$, and  $p \in (1, \infty)$, we consider the eigenvalue problem
\begin{align}\label{eq:PDEint}
\begin{cases}
-\Delta_p u  = \lambda |u|^{p-2} u \, &\text{ in } \mathcal{U}, \\
| \nabla u|^{p-2} \partial_{\nu} u + \alpha |u|^{p-2} u = 0 \, &\text{ on } \partial \mathcal{U},
\end{cases}
\end{align}
where $\Delta_p u := \operatorname{div} \left( | \nabla u|^{p-2} \nabla u \right)$ is the $p$-Laplacian and ${\nu}$ denotes the outer unit normal to $\partial \mathcal{U}$. We understand this problem in the weak sense, meaning $\lambda \in \mathbb{R}$ is called an eigenvalue of \eqref{eq:PDEint}, if there exists a non-zero function $ u \in W^{1,p}(\mathcal{U})$ such that for all $\phi \in W^{1,p}(\mathcal{U})$ it holds
\begin{align}\label{eq:PDEweak}
\int_{\mathcal{U}} | \nabla u |^{p-2} \langle \nabla u , \nabla \phi \rangle \, \mathrm{d}x + \alpha \int_{\partial \mathcal{U}} |u|^{p-2} u \phi \, \mathrm{d}S= \lambda \int_{\mathcal{U}} |u|^{p-2} u \phi \, \mathrm{d}x. 
\end{align}

We are interested in the existence and properties of the first eigenvalue $\lambda_1(\alpha, p,n, \mathcal{U})$ of \eqref{eq:PDEint}, and in when $\lambda_1$ can be characterized variationally as
\begin{align}\label{eq:lambda1int}
\lambda_1(\alpha, p,n, \mathcal{U}) := \inf_{u \in W^{1,p}(\mathcal{U})} \frac{\int_{\mathcal{U}} | \nabla u|^{p} \, \mathrm{d}x + \alpha \int_{\partial \mathcal{U}} |u|^p \, \mathrm{d}S}{\int_{\mathcal{U} } |u|^p \, \mathrm{d}x}.
\end{align}
{In here, by first (or principal) we mean the smallest element of the discrete spectrum if not empty.}

In Section~\ref{sec:bounded}, we consider the case when $\mathcal{U}\equiv\Omega$, where $\Omega$ is a bounded domain. In this setting, several fundamental properties of the $p$-Laplacian have been established by  Lindqvist,  \cite{lindqvist1990equation, lindqvist2017notes}, Anello,  \cite{anello2009dirichlet,anello2013asymptotic} and Lee, \cite{le2006eigenvalue}. It is shown that the first eigenvalue of \eqref{eq:PDEint} can be indeed characterized variationally as in (\ref{eq:lambda1int}),
where the eigenfunctions corresponding to $\lambda_1(\alpha, p,n , \Omega)$ minimize \eqref{eq:lambda1int}, and that it is isolated and simple and has an  eigenfunction of constant sign.

We focus our attention on the asymptotic behavior of  $\lambda_1(\alpha, p,n , \Omega)$ with respect to $\alpha$ and derive quantitative inequalities via a simple, unified approach that can be used for all values of $\alpha$, which is uncommon as the absorbing and the generating cases usually require different strategies of proof. 

For bounded domains, the situation when $\alpha >0$ is well understood. As $\alpha \to \infty$, the first Robin eigenvalue $\lambda_1(\alpha, p,n , \Omega)$ converges to the first Dirichlet eigenvalue of $-\Delta_p$. On the other hand, when $\alpha \searrow 0$ the eigenvalue $\lambda_1(\alpha, p,n , \Omega)$ vanishes and $\frac{\lambda_1(\alpha, p,n , \Omega)}{\alpha}$
tends to $\frac{|\partial \Omega|}{|\Omega|}$. Quantitative inequalities exhibiting  the correct  asymptotic behavior were established by Sperb for $p=2$, \cite{sperb1972untere}, and these results have been recently extended to any $1<p< \infty$ by Barbato and Della Pietra, \cite{barbato2024upper}.

For the superconducting regime, Kova{\v{r}}{\'\i}k and Pankrashkin describe the asymptotic behavior as $\alpha \to - \infty$, depending on the curvature of $\partial \Omega$, \cite{kovavrik2017p}. In fact, in their work $\Omega$ is allowed to be unbounded, as long as is boundary is compact or behaves suitably at infinity. When $\alpha \nearrow 0$ and $\Omega$ is bounded, it still holds $\lambda_1(\alpha, p,n , \Omega) \to 0$ and $\frac{\lambda_1(\alpha, p,n , \Omega)}{\alpha} \to \frac{|\partial \Omega|}{|\Omega|}$, \cite{barbato2024upper}. For $p=2$, corresponding quantitative inequalities are shown in \cite{giorgi2007eigenvalue}, but quantitative inequalities were unknown for $p \neq 2$. In Theorem~\ref{theo:lowbound1}, by adapting the approach from \cite{sperb1972untere, giorgi2005monotonicity}, we establish an inequality, which is a new result for $p \neq 2$ and an improvement of the existing inequality for $p=2$, see Figure \ref{fig:comp}. 

\begin{figure}
	\centering
	\includegraphics[scale=0.4]{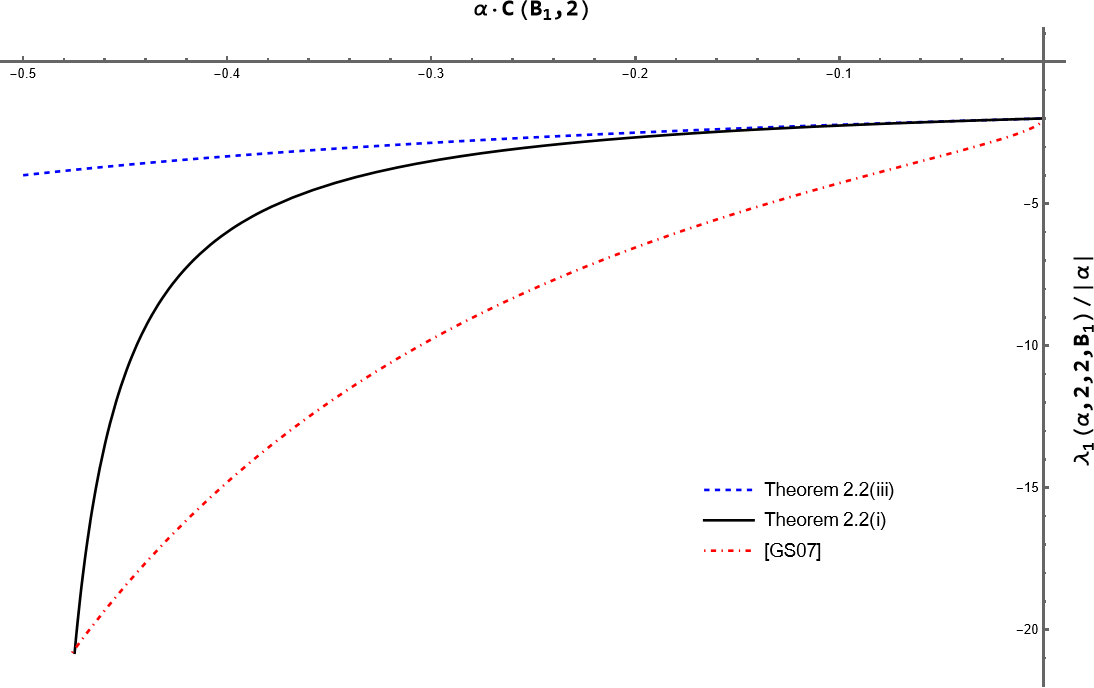}
		\caption{Comparison of the lower bounds of $\lambda_1(\alpha,2,2,B_1)$ for $\alpha <0$}
	\label{fig:comp}
\end{figure}

\begin{customthm}{\ref{theo:lowbound1}}
Let $n \geq 2$, $\Omega \subset \mathbb{R}^n$ be a bounded Lipschitz domain and define
\begin{align*}
C(\Omega,p) := \sup_{\substack{w \in W^{1,p}(\Omega)\\ \int_{\Omega} w \,  \mathrm{d}x= 0}} \frac{\int_{\partial \Omega} |w|^p \, \mathrm{d}S}{\int_{\Omega}| \nabla w|^p \, \mathrm{d}x}.
\end{align*} 
\begin{enumerate}[(i)]
\item For $p\geq 2$ and $\alpha \in \left(\frac{-2^{p-2}}{p C(\Omega,p)}, 0 \right)$,
\begin{align*}
\lambda_1(\alpha,p,n,\Omega) \geq \alpha \frac{|\partial \Omega|}{|\Omega|} \left(1+ 2^{p-2}(p-1)\sqrt[p-1]{\frac{2^{p-2} |\alpha| C(\Omega,p)}{1-2^{p-2}pC(\Omega,p)|\alpha|}} \right).
\end{align*}
\item For  $1 < p \leq 2$ and $\alpha \in \left(\frac{-1}{p C(\Omega,p)}, 0 \right)$,
\begin{align*}
\lambda_1(\alpha,p,n,\Omega) \geq \alpha \frac{|\partial \Omega|}{|\Omega|} \left(1+ (p-1)\sqrt[p-1]{\frac{ |\alpha| C(\Omega,p)}{1-pC(\Omega,p)|\alpha|}} \right).
\end{align*}
\item For $p=2$ and $\alpha \in \left(\frac{-1}{ C(\Omega,2)}, 0 \right)$ we have the stronger inequality
\begin{align*}
\lambda_1(\alpha,2,n,\Omega) \geq \alpha \frac{|\partial \Omega|}{|\Omega|} \left(1+ \frac{ |\alpha| C(\Omega,2)}{1-C(\Omega,2)|\alpha|} \right).
\end{align*}
\end{enumerate}
\end{customthm}
Note that the continuity of the trace operator $T:W^{1,p}(\Omega) \to L^p(\partial \Omega)$ and the Poincaré inequality imply $C(\Omega,p) < \infty$. The approach used to prove Theorem \ref{theo:lowbound1} can also be applied for $\alpha >0$. The resulting inequality does not improve Sperb's result for smooth domains, \cite{sperb1972untere}, but it requires less regularity of $\partial \Omega$ and shows that {our simple method} of proof can be viewed as a unified approach. 

\begin{customthm}{\ref{theo:lowbound2}}
Let $n \geq 2$, $\Omega \subset \mathbb{R}^n$ be a bounded Lipschitz domain and define
\begin{align*}
C_2(\Omega,p) := \sup_{\substack{w \in W^{1,p}(\Omega)\\ \int_{\partial \Omega} w \,  \mathrm{d}S= 0}} \frac{\int_{  \Omega} |w|^p \, \mathrm{d}x}{\int_{\Omega}| \nabla w|^p \, \mathrm{d}x}.
\end{align*} 
\begin{enumerate}[(i)]
\item If $p\geq 2$ and $\alpha \in \left(0, \frac{2^{2-p}|\Omega|}{C_2(\Omega,p) |\partial \Omega|   p} \right)$,
\begin{align*}
\lambda_1(\alpha,p,n,\Omega) \geq \alpha \frac{|\partial \Omega|}{|\Omega|} \left(1- \alpha^\frac{1}{p-1} \left( \frac{C_2(\Omega,p)  |\partial \Omega| \left(  p-1 \right)^{p-1} 2^{(p-2)p} }{|\Omega|-C_2(\Omega,p)\alpha |\partial \Omega|   2^{p-2}p}  \right)^\frac{1}{p-1} \right).
\end{align*}
\item If $1 < p \leq 2$ and $\alpha \in \left(0, \frac{|\Omega|}{C_2(\Omega,p) |\partial \Omega|   p} \right)$,
\begin{align*}
\lambda_1(\alpha,p,n,\Omega) \geq \alpha \frac{|\partial \Omega|}{|\Omega|} \left(1- \alpha^\frac{1}{p-1} \left( \frac{C_2(\Omega,p)  |\partial \Omega| \left(  p-1 \right)^{p-1}  }{|\Omega|-C_2(\Omega,p)\alpha |\partial \Omega|   p}  \right)^\frac{1}{p-1} \right).
\end{align*}
\item If $p=2$ and $\alpha \in \left(0, \frac{|\Omega|}{C_2(\Omega,2) |\partial \Omega|} \right)$,
\begin{align*}
\lambda_1(\alpha,2,n,\Omega) \geq \alpha \frac{|\partial \Omega|}{|\Omega|} \left(1- \frac{c(\alpha)}{2} \right),
\end{align*}
where $c(\alpha) :=  2+\frac{\frac{|\Omega|}{|\partial \Omega|}  -\alpha  C_2(\Omega,2)   }{\alpha   C_2(\Omega,2)  }-\sqrt{\frac{\frac{|\Omega|}{|\partial \Omega|}  -\alpha  C_2(\Omega,2)   }{\alpha   C_2(\Omega,2)  }}\sqrt{4+\frac{\frac{|\Omega|}{|\partial \Omega|}  -\alpha  C_2(\Omega,2)   }{\alpha   C_2(\Omega,2)  }}$.
\end{enumerate}
\end{customthm}
The Poincaré-Wirtinger inequality,  see e.g. \cite[Theorem 4.4.6]{ziemer2012weakly}, implies that the quantity $C_2(\Omega,p)$, which is introduced above, is finite.

In Section~\ref{sec:unbound}, we consider (\ref{eq:PDEint}) for $\mathcal{U}= \Omega^\mathrm{ext}$, where $\Omega$ is a bounded Lipschitz domain and $\Omega$, defined as $\Omega^\mathrm{ext}:= \mathbb{R}^n \setminus \overline{\Omega}$ is its exterior. We will also assume that $\Omega^\mathrm{ext}$ is connected.  In this setting, for  $\lambda_1(\alpha, p,n, \Omega^\mathrm{ext})$ defined as in \eqref{eq:lambda1int}, we first need to ensure that $\lambda_1$ is an eigenvalue, see \cite{krejcirik2016optimisation, krejvcivrik2020optimisation, krejcirik2023optimisation,bundrock2024optimizing} for $p=2$.  {A generalization of these results for $p \neq 2$ can be formulated as follows, see Section~\ref{sec:existence}.}
 
{If $\lambda_1(\alpha, p,n, \Omega^\mathrm{ext}) < 0$, then the infimum in \eqref{eq:lambda1int} is attained by solutions of}
\begin{align}\label{eq:PDE}
\begin{cases}
-\Delta_p u  = \lambda_1(\alpha, p,n, \Omega^\mathrm{ext}) |u|^{p-2} u \, &\text{ in } \Omega^\mathrm{ext}, \\
-  | \nabla u|^{p-2} \partial_\nu u + \alpha |u|^{p-2} u = 0 \, &\text{ on } \partial \Omega,
\end{cases}
\end{align}
where $\nu$ is the unit normal pointing out of $\Omega$. Furthermore,
\begin{enumerate}[(i)]
\item  for $1<p<n$, we have $\lambda_1(\alpha, p,n, \Omega^\mathrm{ext}) < 0$ iff $\alpha < \alpha^*(p,n,\Omega) < 0$, where $\alpha^*(p,n,\Omega)$ is a constant. 
\item for $p\geq n$, we have $\lambda_1(\alpha, p,n, \Omega^\mathrm{ext}) < 0$ iff $\alpha<0$. 
\end{enumerate}
From this we see that considering $\lambda_1(\alpha,p,n,\Omega^\mathrm{ext})$ is not relevant when $\alpha \geq 0$. 

In Section~\ref{sec:ball}, since  even for a ball, it is no longer feasible to explicitly determine the eigenfunctions for $p \neq 2$ and  the methods of proof for $p=2$ tend to relay on explicit formula, we analyze simple geometries and obtain a result that characterizes the asymptotic behavior as $\alpha \nearrow 0$ if $\Omega = B_R$ and $n<p$.

\begin{customthm}{\ref{theo:asym}}
Let $\Omega = B_R \subset \mathbb{R}^n$, where $2 \leq n<p$. For any $\varepsilon > 0$, 
\begin{align*}
\lim_{ \alpha \nearrow 0}
\frac{\lambda_1( \alpha,p,n,\Omega^\mathrm{ext})}{|\alpha|^{\frac{p}{p-n}-\varepsilon}} = 0,
\end{align*}
and
\begin{align*}
\lim_{ \alpha \nearrow 0}
\frac{\lambda_1( \alpha,p,n,\Omega^\mathrm{ext})}{|\alpha|^{\frac{p}{p-n}}} \leq - \left( \frac{ p^n}{2\Gamma(n)} \right)^{\frac{p}{p-n}}.
\end{align*} 
\end{customthm}


In the critical case $p=n$, the behavior of $\lambda_1(\alpha,p,n,B_1^\mathrm{ext})$ can be understood by comparing it to the behavior of $\lambda_1(\alpha,2,2,B_1^\mathrm{ext})$, which can be described using Bessel functions, and it is roughly like ${-\exp({{{2}/{\alpha}}})}$, see Figure \ref{fig:comp2}. We establish this result by applying simple scaling arguments and comparison principles for differential equations.  Analogous bounds can be derived for general $B_R^\mathrm{ext}$.

\begin{figure}
	\centering
	\includegraphics[scale=0.54]{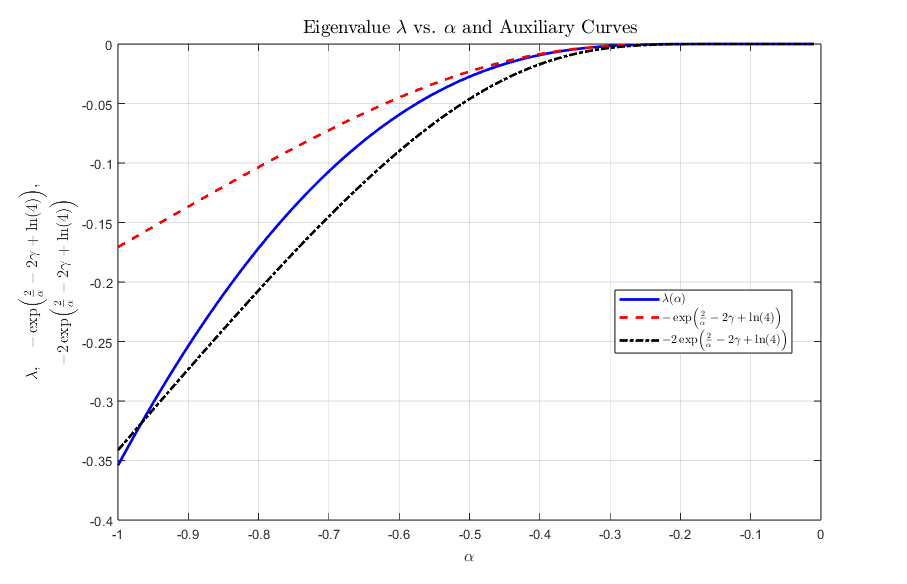}
		\caption{Behavior of $\lambda_1(\alpha,2,2,B_1)$ for $\alpha < 0$}
	\label{fig:comp2}
\end{figure}

\begin{customthm}{\ref{theo:cuteidea}}
Suppose $n\geq 2$ and $\alpha< 0$. Then,
\begin{align*}
|\alpha|^{\frac{n-2}{n-1}} \lambda_{1}(-|\alpha|^{\frac{1} {n-1}},2,2,B_1^{\mathrm{ext}}) \leq \frac{ \lambda_{1}(\alpha,n,n,B_1^{\mathrm{ext}})}{n-1}  \leq - |\lambda_{1}(-|\alpha|^{\frac{1}{n-1}},2,2,B_1^{\mathrm{ext}})|^\frac{n}{2}.
\end{align*}
\end{customthm}

Finally, in Section~\ref{sec:pacwoman} we derive some inequalities for $\lambda_{1}(\alpha,p,n,\Omega^{\mathrm{ext}})$ depending on the geometry of $\Omega$. First, we note that, as for the case $p=2$, the disc is a maximizer in two dimensions. This result can not be extended to higher dimensions. For $n \geq 3$ and $p \in (1, \infty)$, there exists an ellipsoid $E$ with $|E|  = |B_1|$ and
\begin{align*}
\lambda_1 \left( \alpha,p,n,B_1^\mathrm{ext} \right)  < \lambda_1(\alpha,p,n,E^\mathrm{ext})
\end{align*}
for sufficiently negative $\alpha$.

We remark that the minimizing problem of $\lambda_1 \left( \alpha,p,n,\Omega^\mathrm{ext} \right)$ among domains with given measure is only interesting for convex domains, as per theorem below, see Figure~\ref{fig:pacman}.

\begin{figure}
	\centering
	\includegraphics[scale=0.4]{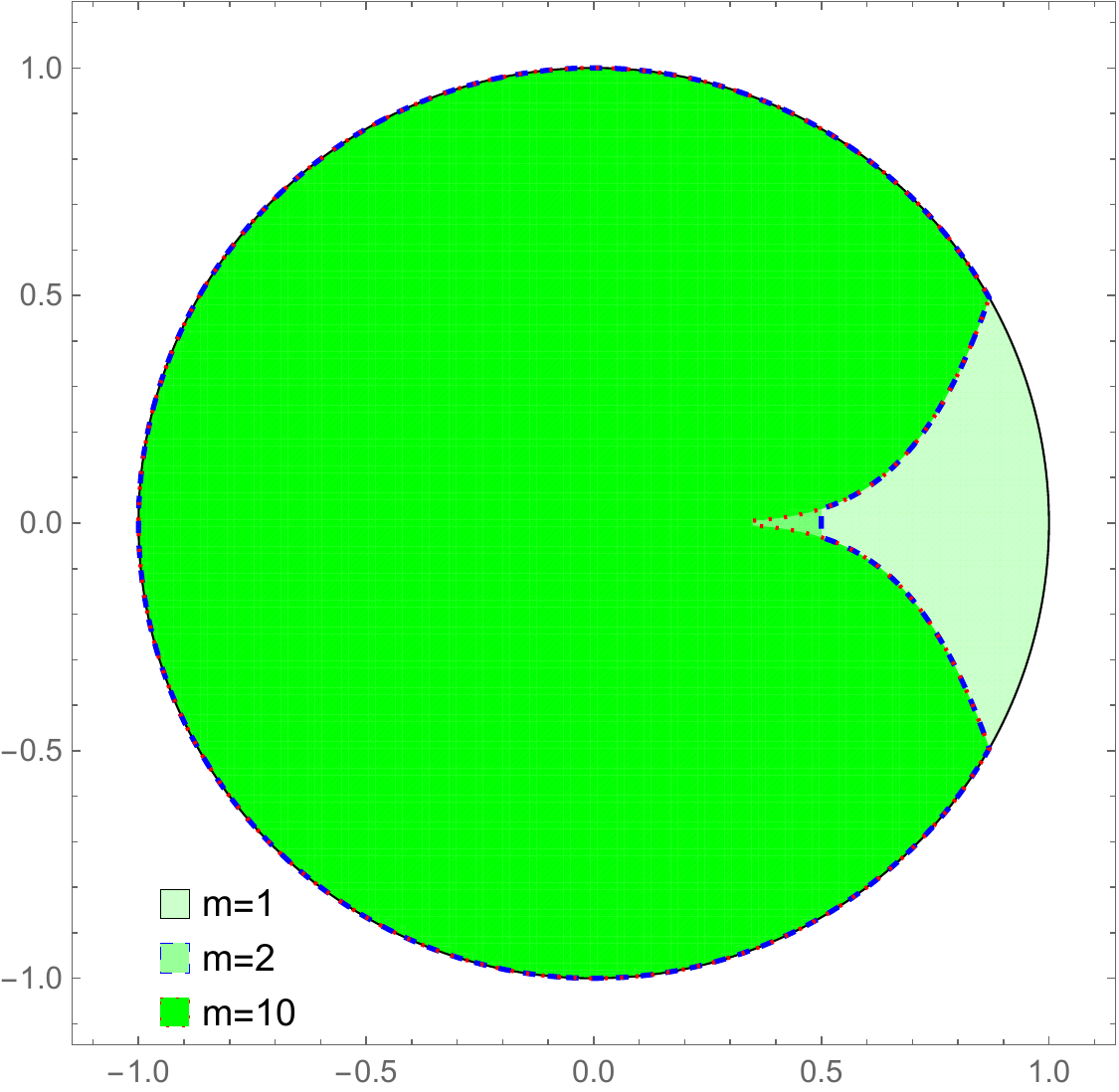}
		\caption{$\Omega_{m}$ from Theorem \ref{prop:bound} for $p=2$}
	\label{fig:pacman}
\end{figure}

\begin{customthm}{\ref{prop:bound}}
For any $\alpha < 0$,
\begin{align*}
\inf \lambda_1(\alpha,2,2,\Omega^\mathrm{ext}) \geq -|\alpha|^2,
\end{align*}
where the infimum is taken over all convex, bounded domains $\Omega \subset \mathbb{R}^2$.

For any $\alpha < 0$ and $p \in (1, \infty)$, there exists a sequence $( \Omega_m )_{m \in \mathbb{N}} \subset \mathbb{R}^n$ (not convex) with $\partial \Omega \in \mathcal{C}^{0,1}$ and $|\Omega_m|, |\partial \Omega_m | < C$ such that
\begin{align*}
\lim_{m \to \infty} \lambda_1(\alpha , p ,2, \Omega_m^\mathrm{ext}) = - \infty.
\end{align*}
\end{customthm}

\section{Eigenvalue Estimates on Bounded Domains}\label{sec:bounded}

\subsection{Estimates for $\alpha<0$}\label{sec:alpha<0}
In here, we take $\mathcal{U}$ to be a bounded domain $\Omega$. As mentioned in the introduction, for $\alpha > 0$ and $p=2$ bounds for $\lambda_1(\alpha,p,n,\Omega)$ have been found by Sperb, \cite{sperb1972untere}. Following Sperb's idea of decomposing $W^{1,2}(\Omega)$ into two function spaces, a lower bound for $\lambda_1(\alpha,p,n,\Omega)$ when $p=2$, $\alpha < 0$
is derived in \cite{giorgi2007eigenvalue}. For $p \neq 2$, the asymptotic behavior of $\lambda_1(\alpha,p,n,\Omega)$ when $\alpha \nearrow 0$ is studied by Barbato and Della Pietra, \cite{barbato2024upper}. In Theorem \ref{theo:lowbound1}, we establish a quantitative inequality for any $1<p<\infty$ and $\alpha<0$, which  improves the result in \cite{giorgi2007eigenvalue} for $p=2$.

\phantomsection
\refstepcounter{theorem} 
\label{theo:lowbound1}

\begin{proof}[Proof of Theorem \ref{theo:lowbound1}]
For any $u \in W^{1,p}(\Omega)$, we define $h := \frac{1}{|\Omega|} \int_{\Omega} u \, \mathrm{d}x$ and $w:= u-h$. Hence, $\int_{\Omega} w \, \mathrm{d}x = 0$. By Jensen's inequality with the probability measure $\mathrm{d} \mu = \frac{1}{|\Omega|} \, \mathrm{d} x$, we have
\begin{align}\label{eq:Jensen}
|\Omega| |h|^p = |\Omega| \left| \int_{\Omega} u(x) \, \mathrm{d} \mu \right|^p \leq |\Omega| \int_{\Omega} |u(x)|^p \, \mathrm{d} \mu = \int_{\Omega} |u(x)|^p \, \mathrm{d} x.
\end{align}
Starting with \eqref{eq:lambda1int} and decomposing any function as described above, we obtain
\begin{align*}
\lambda_1(\alpha, p , \Omega) &= \inf_{\substack{w \in W^{1,p}(\Omega),\\\int_{\Omega} w \, \mathrm{d}x = 0,\\ h \in \mathbb{R}}} \frac{\int_{\Omega} | \nabla w|^{p} \, \mathrm{d}x + \alpha \int_{\partial \Omega} |w+h|^p \, \mathrm{d}S}{\int_{\Omega } |w+h|^p \, \mathrm{d}x}.
\end{align*}
To estimate the boundary integral, we will use Lemma \ref{lemma:basicineq} \eqref{eq:ungl1} and \ref{lemma:basicineq}\eqref{eq:ungl2} for $p \geq 2$ and $p \leq 2$ respectively. Only the proof for $p \geq 2$ is presented here, since the proofs of both cases do not differ except for a constant.

Suppose $p \geq 2$ and $\alpha \in \left(\frac{-2^{p-2}}{p C(\Omega,p)}, 0 \right)$.  Using Lemma \ref{lemma:basicineq}\eqref{eq:ungl1}, we have
\begin{align*}
\int_{\partial \Omega} |w+h|^p \, \mathrm{d}S \leq \int_{\partial \Omega}|h|^p + 2^{p-2} p \left( |w|^p + |w| |h|^{p-1} \right) \, \mathrm{d}S,
\end{align*}
and therefore $\lambda_1(\alpha, p,n , \Omega)$ is greater than
\begin{align*}
& \inf_{\substack{w \in W^{1,p}(\Omega),\\\int_{\Omega} w \, \mathrm{d}x = 0,\\ h \in \mathbb{R}}} \frac{\int_{\Omega} | \nabla w|^{p} \, \mathrm{d}x + \alpha \int_{\partial \Omega}|h|^p + 2^{p-2} p \left( |w|^p + |w| |h|^{p-1} \right) \, \mathrm{d}S}{\int_{\Omega } |w+h|^p \, \mathrm{d}x}\\
\geq \,&  \inf_{\substack{w \in W^{1,p}(\Omega),\\\int_{\Omega} w \, \mathrm{d}x = 0,\\ h \in \mathbb{R}}} \alpha \frac{ | \partial \Omega||h|^p}{|\Omega| |h|^p}  + \frac{\int_{\Omega} | \nabla w|^{p} \, \mathrm{d}x + \alpha 2^{p-2} p \int_{\partial \Omega}   |w|^p + |w| |h|^{p-1}  \, \mathrm{d}S}{\int_{\Omega } |w+h|^p \, \mathrm{d}x}
\end{align*}
where we used \eqref{eq:Jensen} in the second step. 
Hence, it remains to show that for any $w \in W^{1,p}(\Omega)$ with $\int_{\Omega} w \, \mathrm{d}x = 0$ and any $h \in \mathbb{R}$,
\begin{align*}
& \frac{\int_{\Omega} | \nabla w|^{p} \, \mathrm{d}x + \alpha 2^{p-2} p \int_{\partial \Omega}   |w|^p + |w| |h|^{p-1}  \, \mathrm{d}S}{\int_{\Omega } |w+h|^p \, \mathrm{d}x} \geq \alpha \frac{|\partial \Omega|}{|\Omega|}  2^{p-2}(p-1)\sqrt[p-1]{\frac{2^{p-2} |\alpha| C(\Omega,p)}{1-2^{p-2}pC(p,\Omega)|\alpha|}}.
\end{align*}
For convenience, we introduce the notation
\begin{align*}
A:= \alpha 2^{p-2} p \int_{\partial \Omega}   |w|   \, \mathrm{d}S, \quad \quad B:= \alpha |\partial \Omega|  2^{p-2}(p-1)\sqrt[p-1]{\frac{2^{p-2} |\alpha| C(\Omega,p)}{1-2^{p-2}pC(p,\Omega)|\alpha|}}.
\end{align*}
Then, after multiplication with  $\int_{\Omega } |w+h|^p \, \mathrm{d}x$, the inequality that has to be shown can be written as
\begin{align*}
\int_{\Omega} | \nabla w|^{p} \, \mathrm{d}x + \alpha 2^{p-2} p \int_{\partial \Omega}   |w|^p   \, \mathrm{d}S+ A |h|^{p-1}  \geq  \frac{B}{|\Omega|}\int_{\Omega } |w+h|^p \, \mathrm{d}x.
\end{align*}
In view of  \eqref{eq:Jensen}, $\int_{\Omega } |w+h|^p \, \mathrm{d}x  \geq |\Omega| |h|^p$. Hence, it remains to show
\begin{align}\label{eq:ineqAB}
\int_{\Omega} | \nabla w|^{p} \, \mathrm{d}x + \alpha 2^{p-2} p \int_{\partial \Omega}   |w|^p   \, \mathrm{d}S+ A |h|^{p-1} -B |h|^p \geq 0.	
\end{align}
For $a,b \in \R$, $a,b<0$, the function $f:[0,\infty) \to \mathbb{R}$, $f(x) := a x^{p-1}-b x^p$, attains its minimum at $x^*= \frac{a(p-1)}{bp}>0$ with
\begin{align*}
f \left( x^* \right) = a \left( \frac{a(p-1)}{bp} \right)^{p-1}-b \left( \frac{a(p-1)}{bp} \right)^p = \frac{a}{p}  \left( \frac{a(p-1)}{bp} \right)^{p-1}.
\end{align*}
With $a=A$, $b=B$ and $x=|h|$, we get
\begin{align*}
&\int_{\Omega} | \nabla w|^{p} \, \mathrm{d}x + \alpha 2^{p-2} p \int_{\partial \Omega}   |w|^p   \, \mathrm{d}S+ A |h|^{p-1} -B |h|^p\\
 \geq\, &\int_{\Omega} | \nabla w|^{p} \, \mathrm{d}x + \alpha 2^{p-2} p \int_{\partial \Omega}   |w|^p   \, \mathrm{d}S+  \alpha 2^{p-2}   \frac{ \left( \int_{\partial \Omega}   |w|   \, \mathrm{d}S \right)^p}{ |\partial \Omega|^{p-1}   \frac{2^{p-2} |\alpha| C(\Omega,p)}{1-2^{p-2}pC(p,\Omega)|\alpha|} }
\end{align*}
where we inserted the definition of $A$ and $B$. By H\"older's inequality we have
\begin{align*}
\int_{\partial\Omega} |w|\,\mathrm{d}S \leq \left(\int_{\partial\Omega} |w|^p\,\mathrm{d}S\right)^{\frac{1}{p}} \left(|\partial\Omega|\right)^{\frac{1}{q}}
\end{align*}
where $\frac{1}{p}+\frac{1}{q}=1$. Raising both sides to the power $p$ and considering the definition of $C(\Omega,p)$, gives
\begin{align*}
\left(\int_{\partial\Omega} |w|\,\mathrm{d}S\right)^p  \leq |\partial\Omega|^{p-1} C(\Omega,p) \int_{\Omega} |\nabla w|^p\,\mathrm{d}x.
\end{align*}
Therefore,
\begin{align*}
&\int_{\Omega} | \nabla w|^{p} \, \mathrm{d}x + \alpha 2^{p-2} p \int_{\partial \Omega}   |w|^p   \, \mathrm{d}S+  \alpha 2^{p-2}   \frac{ \left( \int_{\partial \Omega}   |w|   \, \mathrm{d}S \right)^p}{ |\partial \Omega|^{p-1}   \frac{2^{p-2} |\alpha| C(\Omega,p)}{1-2^{p-2}pC(p,\Omega)|\alpha|} } \\
 \geq \, &\int_{\Omega} | \nabla w|^{p} \, \mathrm{d}x \left( 1+ \alpha 2^{p-2} p C(\Omega,p)+ \alpha 2^{p-2}   \frac{  C(\Omega,p) }{    \frac{2^{p-2} |\alpha| C(\Omega,p)}{1-2^{p-2}pC(p,\Omega)|\alpha|} } \right) = 0.
\end{align*}
Hence, \eqref{eq:ineqAB} holds true and this completes the proof. The improvement for $p=2$ can be obtained by writing $(a+b)^2=b^2+2ab+a^2$ instead of using Lemma \ref{lemma:basicineq} \eqref{eq:ungl1}.
\end{proof}

Compared to \cite{giorgi2007eigenvalue}, we obtain a better lower bound when $\alpha$ is small. However, our inequality is valid for a smaller range of $\alpha$. With Lemma \ref{lemma:basicineq}, we can proceed as in \cite{giorgi2007eigenvalue} also for $p\neq 2$, which leads to the following result.

\begin{proposition}\label{lemma:upbound}
Let $\Omega \subset \mathbb{R}^n$, $n \geq 2$, be a bounded Lipschitz domain and define
\begin{align*}
C(\Omega,p) := \sup_{\substack{w \in W^{1,p}(\Omega)\\ \int_{\Omega} w \,  \mathrm{d}x= 0}} \frac{\int_{\partial \Omega} |w|^p \, \mathrm{d}S}{\int_{\Omega}| \nabla w|^p \, \mathrm{d}x}.
\end{align*} 
\begin{enumerate}[(i)]
\item For $1 < p \leq 2$ and $\alpha \in \left( \frac{-1}{C(\Omega,p)} ,0\right)$,
\begin{align*}
 \lambda_1(\alpha,p,n,\Omega) \geq \alpha \frac{|\partial \Omega|}{ |  \Omega|} \left( 1- \sqrt{C(\Omega,p)|\alpha|} \right)^{-2}.
\end{align*}
\item For $p \geq 2$ and $\alpha \in \left( -\frac{2^{2-p}}{C(\Omega,p)} ,0\right)$,
\begin{align*}
 \lambda_1(\alpha,p,n,\Omega) \geq \alpha \frac{|\partial \Omega|}{ |  \Omega|} \left( \sqrt{\frac{1}{2^{p-2}}}- \sqrt{C(\Omega,p)|\alpha|} \right)^{-2}.
\end{align*}
\end{enumerate}
\end{proposition}
\begin{proof}
We only write down the proof for $p \geq 2$. In the same manner, the statement can be shown for $1<p \leq 2$. Following \cite{giorgi2007eigenvalue}, for any $\lambda > 0$,
\begin{align}\label{eq:lowbound}
\frac{1}{-\lambda_1(\alpha,p,n,\Omega)} \geq \frac{\widehat{C}(\lambda, \Omega)}{|\alpha|} - \lambda,
\end{align}
where
\begin{align*}
\widehat{C}(\lambda, \Omega) := \inf_{  u \in W^{1,p}(\Omega)} \frac{\int_\Omega |  u|^p \, \mathrm{d}x + \lambda \int_\Omega |\nabla u|^p \, \mathrm{d}x }{  \int_{\partial \Omega} |u|^p \, \mathrm{d}S}.
\end{align*}
In order to find a lower bound for $\widehat{C}(\lambda, \Omega)$, we consider
\begin{align*}
\frac{1}{\widehat{C}(\lambda, \Omega)} = \sup_{  u \in W^{1,p}(\Omega)} \frac{  \int_{\partial \Omega} |u|^p \, \mathrm{d}S}{\int_\Omega |  u|^p \, \mathrm{d}x + \lambda \int_\Omega |\nabla u|^p \, \mathrm{d}x }.
\end{align*}
Any $u \in W^{1,p}(\Omega)$ may be written as $u = w+h$ where $h=\frac{1}{|\Omega|}\int_\Omega u \, \mathrm{d}x$ and $\int_\Omega w \, \mathrm{d}x = 0$. Together with the triangle inequality this yields
\begin{align*}
\frac{1}{\widehat{C}(\lambda, \Omega)} \leq  \sup_{  w \in W^{1,p}(\Omega), \int_\Omega w \, \mathrm{d}x = 0, h \in \mathbb{R}} \frac{ \left( \sqrt[p]{ \int_{\partial \Omega} |w|^p \, \mathrm{d}S} + \sqrt[p]{\int_{\partial \Omega} |h|^p \, \mathrm{d}S} \right)^p}{\int_\Omega |  w+h|^p \, \mathrm{d}x + \lambda \int_\Omega |\nabla w|^p \, \mathrm{d}x }.
\end{align*} 
Using Lemma \ref{lemma:basicineq}\eqref{eq:algineq} (for $1<p\leq 2$ use Lemma \ref{lemma:basicineq}\eqref{eq:algineqimp} instead), we obtain
\begin{align*}
\frac{2^{2-p}}{\widehat{C}(\lambda, \Omega)} & \leq  \sup_{ \substack{ w \in W^{1,p}(\Omega),\\ \int_\Omega w \, \mathrm{d}x = 0,\\ h \in \mathbb{R}}} \frac{   \int_{\partial \Omega} |h|^p \, \mathrm{d}S}{  \int_\Omega |w+h|^p \, \mathrm{d}x } +\frac{ \int_{\partial \Omega} |w|^p \, \mathrm{d}S}{ \lambda \int_\Omega |\nabla w|^p \, \mathrm{d}x } \leq   \frac{  | \partial \Omega|}{  |\Omega| } +\frac{ C(\Omega,p)}{ \lambda } 
\end{align*}
where we used \eqref{eq:Jensen} and the definition of $C(\Omega,p)$. Hence, \eqref{eq:lowbound} yields
\begin{align*}
\frac{1}{-\lambda_1(\alpha,p,n,\Omega)} \geq \frac{2^{2-p}}{|\alpha|} \frac{1}{\frac{|\partial \Omega|}{| \Omega|} + \frac{C(\Omega,p)}{\lambda} } - \lambda
\end{align*}
for any $\lambda>0$ and any $\alpha < 0$. For $|\alpha|<\frac{2^{2-p}}{C(\Omega,p)}$, we choose $\lambda = \frac{\sqrt{\frac{C(\Omega,p)}{2^{p-2}|\alpha|}}-C(\Omega,p)}{\frac{|\partial \Omega|}{|\Omega|}}$
and obtain
\begin{align*}
\frac{1}{-\lambda_1(\alpha,p,n,\Omega)} 
&= \frac{ |\Omega|}{|\alpha||\partial \Omega|}  \left( \sqrt{\frac{1}{2^{p-2}}}- \sqrt{C(\Omega,p)|\alpha|} \right)^2.
\end{align*}
Rearranging this inequality yields the claimed statement.
\end{proof}


\subsection{Estimates for $\alpha>0$}
For $\alpha > 0$, \cite{sperb1972untere, barbato2024upper} derive inequality with the help of Thomson's principle for $\partial \Omega \in \mathcal{C}^{1, \gamma}$. In Theorem \ref{theo:lowbound2}, using the idea presented in Section \ref{sec:alpha<0}, we prove an inequality for Lipschitz domains.

\refstepcounter{theorem} 
\label{theo:lowbound2}
\begin{proof}[Proof of Theorem \ref{theo:lowbound2}]
For any $u \in W^{1,p}(\Omega)$, we define $h := \frac{1}{|\partial \Omega|} \int_{\partial \Omega} u \, \mathrm{d}S$ and $w:= u-h$. Hence, $\int_{\partial \Omega} w \, \mathrm{d}S = 0$. By Jensen's inequality with the probability measure $\mathrm{d} \mu = \frac{1}{| \partial \Omega|} \, \mathrm{d} S$, we have
\begin{align}\label{eq:Jensen1}
| \partial\Omega| |h|^p = | \partial\Omega| \left| \int_{ \partial\Omega} u \, \mathrm{d} \mu \right|^p \leq | \partial \Omega| \int_{ \partial\Omega} |u|^p \, \mathrm{d} \mu = \int_{ \partial \Omega} |u|^p \, \mathrm{d} S.
\end{align}
Starting with \eqref{eq:lambda1int} and decomposing any function as described above, we obtain
\begin{align*}
\lambda_1(\alpha, p , \Omega) &\geq \inf_{\substack{w \in W^{1,p}(\Omega)\\\int_{\partial \Omega} w \, \mathrm{d}S = 0\\ h \in \mathbb{R}}} \frac{\int_{\Omega} | \nabla w|^{p} \, \mathrm{d}x + \alpha |\partial \Omega| |h|^p}{\int_{\Omega } |w+h|^p \, \mathrm{d}x}.
\end{align*}
To estimate the denominator, we use Lemma \ref{lemma:basicineq}\eqref{eq:ungl1} for $p \geq 2$ and Lemma \ref{lemma:basicineq}\eqref{eq:ungl2} for $1< p \leq 2$. Since the two cases work analogously, we only carry out the details for $p \geq 2$. We have
\begin{align*}
\lambda_1(\alpha, p , \Omega) &\geq \inf_{\substack{w \in W^{1,p}(\Omega)\\\int_{\partial \Omega} w \, \mathrm{d}S = 0\\ h \in \mathbb{R}}} \frac{\int_{\Omega} | \nabla w|^{p} \, \mathrm{d}x + \alpha |\partial \Omega| |h|^p}{\int_{\Omega } |h|^p + 2^{p-2}p (|w|^p+|w||h|^{p-1}) \, \mathrm{d}x}.
\end{align*}
Thus, the claimed inequality is shown if, for any $h \in \R$ and any $w \in W^{1,p}(\Omega)$ with $\int_{\partial \Omega} w \, \mathrm{d}S = 0$,
\begin{align*}
\frac{\int_{\Omega} | \nabla w|^{p} \, \mathrm{d}x + \alpha |\partial \Omega| |h|^p}{\int_{\Omega } |h|^p + 2^{p-2}p (|w|^p+|w||h|^{p-1}) \, \mathrm{d}x} \geq \alpha \frac{|\partial \Omega|}{|\Omega|} \left(1- c(\alpha) \right),
\end{align*}
or equivalently,
\begin{align}
&\int_{\Omega} | \nabla w|^{p} \, \mathrm{d}x  -\alpha \frac{|\partial \Omega|}{|\Omega|} \left(1- c(\alpha) \right)  2^{p-2}p \int_{\Omega }  |w|^p \, \mathrm{d}x \nonumber \\
 \geq \, &-\alpha |\partial \Omega|  c(\alpha)   |h|^p  +\alpha \frac{|\partial \Omega|}{|\Omega|} \left(1- c(\alpha) \right)  2^{p-2}p \int_{\Omega }  |w||h|^{p-1} \, \mathrm{d}x. \label{eq:toshow1}
\end{align}
If $c(\alpha) \geq 1$, the inequity is obviously true, so we may assume  $c(\alpha) < 1$. 
As a first step, we choose $h$ such that the right hand side becomes maximal. To this end, we introduce the notation,
\begin{align*}
A:=-\alpha |\partial \Omega|  c(\alpha), \quad \quad B:=\alpha \frac{|\partial \Omega|}{|\Omega|} \left(1- c(\alpha) \right)  2^{p-2}p \int_{\Omega }  |w| \, \mathrm{d}x.
\end{align*}
Note that $A<0$ and $B>0$ if $c(\alpha)<1$. It is straight forward to show $ A x^p + B x^{p-1} \leq  \left(  \frac{1-p}{A} \right)^{p-1} \frac{B^p}{p^p}$. Thus,
\begin{align*}
&-\alpha |\partial \Omega|  c(\alpha)   |h|^p  +\alpha \frac{|\partial \Omega|}{|\Omega|} \left(1- c(\alpha) \right)  2^{p-2}p \int_{\Omega }  |w||h|^{p-1} \, \mathrm{d}x\\
\leq \,& \alpha c(\alpha) \frac{|\partial \Omega|}{|\Omega|^p} \left(  p-1 \right)^{p-1} 2^{(p-2)p} |\Omega|^{p-1} \int_{\Omega} |w|^p\,\mathrm{d}x \left( \frac{1-c(\alpha)}{c(\alpha)} \right)^p,
\end{align*}
where we used H{\"o}lder's inequality. For the left hand side of \eqref{eq:toshow1}, we can use the definition of $C(\Omega, p)$ and obtain
\begin{align*}
\int_{\Omega} | \nabla w|^{p} \, \mathrm{d}x  -\alpha \frac{|\partial \Omega|}{|\Omega|} \left(1- c(\alpha) \right)  2^{p-2}p \int_{\Omega }  |w|^p \, \mathrm{d}x  \, & \frac{\int_{ \Omega} |  w|^{p} \, \mathrm{d}x}{C(\Omega,p)}  -\alpha \frac{|\partial \Omega|}{|\Omega|}   2^{p-2}p \int_{\Omega }  |w|^p \, \mathrm{d}x.
\end{align*}
Hence, the claimed inequality is proven if
\begin{align*}
 &\frac{1}{C(\Omega,p)}  -\alpha \frac{|\partial \Omega|}{|\Omega|}   2^{p-2}p \geq  \alpha  \frac{|\partial \Omega|}{|\Omega|} \left(  p-1 \right)^{p-1} 2^{(p-2)p}  \left( \frac{1-c(\alpha)}{c(\alpha)} \right)^p c(\alpha).
\end{align*}
which is equivalent to
\begin{align}\label{eq:ineqc(alpha)}
 \left( \frac{\frac{|\Omega|}{C(\Omega,p)}  -\alpha |\partial \Omega|   2^{p-2}p}{\alpha  |\partial \Omega| \left(  p-1 \right)^{p-1} 2^{(p-2)p}} \right)^\frac{1}{p} \geq     \frac{1-c(\alpha)}{c(\alpha)}  c(\alpha)^\frac{1}{p} = c(\alpha)^{\frac{1-p}{p}}-c(\alpha)^{\frac{1}{p}}.
\end{align}
Thus, if we chose
\begin{align*}
c(\alpha):= \alpha^\frac{1}{p-1} \left( \frac{C(\Omega,p)  |\partial \Omega| \left(  p-1 \right)^{p-1} 2^{(p-2)p} }{|\Omega|-C(\Omega,p)\alpha |\partial \Omega|   2^{p-2}p}  \right)^\frac{1}{p-1}
\end{align*}
and use $c(\alpha)^{\frac{1-p}{p}}-c(\alpha)^{\frac{1}{p}} \leq c(\alpha)^{\frac{1-p}{p}}$, the claimed inequality holds true.

The improvement for $p=2$ can be obtained by using $(a+b)^2=b^2+2ab+b^2$ instead of Lemma \ref{lemma:basicineq}\eqref{eq:ungl1} and solving \eqref{eq:ineqc(alpha)} explicitly.
\end{proof}

%

\section{Exterior Domains}\label{sec:unbound}

In this section, we study \eqref{eq:PDEint} with $\mathcal{U}=\Omega^\mathrm{ext}$, where $\Omega^\mathrm{ext} := \mathbb{R}^n \setminus \overline{\Omega}$ for a bounded Lipschitz domain $\Omega \subset \mathbb{R}^n$. We also assume $\Omega^\mathrm{ext}$ to be connected, note that this does not require $\Omega$ itself to be connected. 
As mentioned in the introduction, for $p=2$, related problems were explored in \cite{krejcirik2016optimisation,krejvcivrik2020optimisation,krejcirik2023optimisation,kachmar2025laplace}. 
Krej{\v{c}}i{\v{r}}{\'\i}k and Lotoreichik show for $p=2$ that the associated operator has a nonempty essential spectrum given by $ [0, \infty)$. Hence, if $\lambda_1(\alpha, 2,n, \Omega^\mathrm{ext}) < 0$, then it is part of the discrete spectrum and there exists  a corresponding eigenfunction. They also prove the existence of a constant $\alpha^*(2,n,\Omega^\mathrm{ext}) \leq 0$ such that $\lambda_1(\alpha, 2,n, \Omega^\mathrm{ext})<0$ if and only if $\alpha < \alpha^*(2,n,\Omega^\mathrm{ext})$. Furthermore, $\alpha^*(2,n,\Omega^\mathrm{ext})$ coincides with the first harmonic Steklov eigenvalue, \cite{bundrock2023robin,bundrock2024optimizing}.

\subsection{Existence of a Variational First Eigenvalue}\label{sec:existence}

For $p \neq 2$, the $p$-Laplacian is not linear, rendering the usual spectral theory inapplicable. Nevertheless, we have a similar behavior. Lemma \ref{lemma:nonpositiv} mirrors the fact that the lowest point of the essential spectrum of the $2$-Laplacian is zero.

\begin{lemma}\label{lemma:nonpositiv}
For a bounded domain $\Omega \subset \mathbb{R}^n$, $n \geq 2$ and $ p \in (1, \infty)$, one has $\lambda_1(\alpha,p,n,\Omega^\mathrm{ext}) \leq 0$ for any $\alpha \in \mathbb{R}$.
\end{lemma}
\begin{proof}
Let $\phi \in \mathcal{C}^\infty_0(\mathbb{R}^n)$ satisfy $\int_{\mathbb{R}^n} |\phi|^p \, \mathrm{d}x = 1$ and $0 \neq x_0 \in \mathbb{R}^n$. For $m \in \mathbb{N}$, we define the sequence $(\phi_m)_{m \in \mathbb{N}}$, $\phi_m : \mathbb{R}^n \to \mathbb{R}$ by
\begin{align*}
\phi_m(x) := \frac{1}{m^\frac{n}{p}}\, \phi\left( \frac{x-m^2 x_0}{m} \right).
\end{align*}
As $\phi$ has compact support, $\operatorname{supp}(\phi_m) \subset \Omega^\mathrm{ext}$ for sufficiently large $m$. 
Thus, it is straight forward to show $\lim_{m \to \infty} \frac{\int_{\Omega^\mathrm{ext}} | \nabla \phi_m|^{p} \, \mathrm{d}x + \alpha \int_{\partial \Omega} |\phi_m|^p \, \mathrm{d}S}{\int_{\Omega^\mathrm{ext}} |\phi_m|^p \, \mathrm{d}x}  = 0$, implying Lemma \ref{lemma:nonpositiv}.
\end{proof}

If $\alpha \geq 0$, clearly $\lambda_1(\alpha,p,n,\Omega^\mathrm{ext}) = 0$. We observe that $\lambda_1$ may also vanish for $\alpha<0$. Nonetheless, if $\lambda_1$ is strictly negative, we can ensure the existence of a minimizer for \eqref{eq:lambda1int}.

\begin{lemma}\label{theo:existencemin}
Suppose that $\Omega \subset \mathbb{R}^n$, $n \geq 2$  is a bounded  Lipschitz domain and that $ p \in (1, \infty)$. If $\lambda_1(\alpha, p,n,\Omega^\mathrm{ext}) < 0$, then $\lambda_1(\alpha, p,n,\Omega^\mathrm{ext})$ is an eigenvalue of \eqref{eq:PDE}, i.e. there exists a function $u \in W^{1,p}(\Omega^\mathrm{ext})$ solving \eqref{eq:PDEweak}.
\end{lemma}
\begin{proof}
We provide only a sketch of the proof since the result follows by well-known arguments. Let $(u_m)_{m \in \mathbb{N}} \subset W^{1,p}(\Omega^\mathrm{ext})$ be a sequence minimizing \eqref{eq:lambda1int} with $|| u_m||_{L^p(\Omega^\mathrm{ext})} = 1$. Then, $|| u_m||_{W^{1,p}(\Omega^\mathrm{ext})}$ is bounded, which implies the existence of a weakly convergent subsequence, with weak limit $u^* \in W^{1,p}(\Omega^\mathrm{ext})$ satisfying
\begin{align*}
\lambda_1(\alpha,p,n,\Omega^\mathrm{ext}) \geq  \int_{\Omega^\mathrm{ext}} | \nabla u^*|^p \, \mathrm{d}x  + \alpha \int_{\partial \Omega} | u^*|^p \, \mathrm{d}S.
\end{align*}
Thus, $\int_{\Omega^\mathrm{ext}} |  u^*|^p \, \mathrm{d}x =1$. Hence, $u^* \in W^{1,p}(\Omega^\mathrm{ext})$ is a minimizer of \eqref{eq:lambda1int}. Standard methods establish that this minimizer serves as a weak solution.
\end{proof}

If $\lambda_1(\alpha,p,n,\Omega^\mathrm{ext})=0$,  we can still derive the existence of a weakly convergent subsequence as described above. However, we cannot guarantee $u^* \not\equiv 0$.
%

As the negativity of $\lambda_1$ is crucial to ensure the existence of an eigenfunction, we characterize its dependence on $\alpha$. For $p \geq n$, we can simply follow the proof of \cite[Proposition 2]{krejcirik2016optimisation}.

\begin{lemma}\label{lemma:alpha0}
Let $2 \leq n \leq p < \infty$ and let $\Omega \subset \mathbb{R}^n$ be a bounded domain. Then,
\begin{align*}
\lambda_1(\alpha, p,n,\Omega^\mathrm{ext}) < 0 \quad \Leftrightarrow \quad  \alpha < 0.
\end{align*}
\end{lemma}

However, $\alpha<0$ does not guarantee $\lambda_1(\alpha , p , n,\Omega^\mathrm{ext}) < 0$ if $p < n$. A similar observation was made for the case $p=2$ and $n \geq 3$ in \cite{krejcirik2016optimisation}. In order to characterize the negativity of $\lambda_1(\alpha,p,n,\Omega^\mathrm{ext})$, we consider the $p$-harmonic Steklov eigenvalue problem discussed by Han in \cite{han2016first}, 
\begin{align}\label{eq:PDEStek}
\begin{cases}
\Delta_p u = 0 \, &\text{ in } \Omega^\mathrm{ext}, \\
-  |\nabla u |^{p-2}  \partial_\nu u = \mu  |u|^{p-2} u \, &\text{ on } \partial \Omega.
\end{cases}
\end{align}
If $\Omega$ is a ball  and $p \in (1,n)$, a radial solution of \eqref{eq:PDEStek} takes the form
\begin{align*}
u(x) = c_1 + c_2 |x|^{-\frac{n-p}{p-1}}, \quad c_1, c_2 \in \mathbb{R}.
\end{align*}
Setting $c_1=0$ and $c_2=1$, this function decays at infinity. However, $u \notin L^p(B_R^\mathrm{ext})$ for $\sqrt{n} \leq p$. Consequently, the solution of \eqref{eq:PDE} might not belong to $W^{1,p}(\Omega^\mathrm{ext})$. To address this issue, Auchmuty and Han introduce the space of finite $p$-energy functions, denoted by $E^{1,p}(\Omega^{\mathrm{ext}})$, \cite[Section 3]{auchmuty2014p}. Although they initially consider $n \geq 3$, they note that the space is also well-defined for $n=2$ and $p \in (1,2)$, as discussed in \cite[p. 264]{auchmuty2014p}. This follows from Sobolev embedding results, found in \cite{lieb2001analysis} for $\mathbb{R}^n$, with best constants provided by Talenti in \cite{talenti1976best}.  For $n \geq 2$, $p \in (1,n)$, and a bounded domain $\Omega \subset \mathbb{R}^n$ with Lipschitz boundary, a function $u \in E^{1,p}(\Omega^{\mathrm{ext}})$ if 
\begin{enumerate}[(i)]
\item $u \in L^1_{\text{loc}}(\Omega^\mathrm{ext})$,
\item $|\nabla u| \in L^p(\Omega^{\mathrm{ext}};\R)$,
\item  $S_c := \{ x \in \Omega^{\mathrm{ext}}: | u(x) | \geq c \}$ has finite measure for any $c > 0$.
\end{enumerate}
In \cite{han2016first}, for $p \in (1,n)$, it is shown that the first eigenvalue is given by
\begin{align}\label{eq:mu1}
\mu_1(p,n,\Omega^\mathrm{ext}) = \inf_{u \in E^{1,p}(\Omega^\mathrm{ext})} \frac{\int_{\Omega^\mathrm{ext}} | \nabla u |^p \, \mathrm{d}x}{\int_{\partial \Omega} | u |^p \, \mathrm{d}S}.
\end{align}
Moreover, the infimum is attained, and the corresponding eigenfunction serves as the minimizer. Note that for $p\geq n$, one can show that $\mu_1(p,n,\Omega^\mathrm{ext})=0$ using as test functions the sequence of Lemma~\ref{lemma:nonpositiv}.
Defining
$$
\alpha^*(p,n,\Omega^\mathrm{ext}):= -\mu_1(p,n,\Omega^\mathrm{ext}), 
$$
these properties allow us to prove Lemma \ref{lemma:alpha*} below.

\begin{lemma}\label{lemma:alpha*}
Let $2 \leq n$, $p \in (1,n)$, and $\Omega \subset \mathbb{R}^n$ be a bounded domain with Lipschitz boundary. Then, 
\begin{align*}
\lambda_1(\alpha, p,n,\Omega^\mathrm{ext}) < 0 \quad \Leftrightarrow \quad \alpha < \alpha^*(p,n,\Omega^\mathrm{ext}).
\end{align*}
\end{lemma}
\begin{proof}
If $\lambda_1(\alpha, p,n,\Omega^\mathrm{ext}) < 0$, there exists a  $u \in W^{1,p}(\Omega^{\mathrm{ext}})$ with $\int_{\Omega^\mathrm{ext}} | \nabla u |^p \, \mathrm{d}x+ \alpha \int_{\partial \Omega} | u |^p \, \mathrm{d}S <0$. Since $W^{1,p}(\Omega^{\mathrm{ext}}) \subseteq  E^{1,p}(\Omega^{\mathrm{ext}})$, this implies $\mu_1(p,n,\Omega^\mathrm{ext}) <- \alpha $.

To prove the reverse inequality, we approximate the first eigenfunction of \eqref{eq:PDEStek} using functions with compact support. For a function $u \in W^{1,p}_{\text{loc}}(\Omega^{\mathrm{ext}})$ with $| \nabla u | \in L^p(\Omega^{\mathrm{ext}})$, it is shown in \cite[Theorem 1.1, Proposition 2.2]{LuOu}, that  $w := u - \left( u \right)_\infty$ can be approximated by smooth functions, where $ \left( u \right)_\infty = 0$ for any $u \in E^{1,p}(\Omega^{\mathrm{ext}})$. Specifically,  for every $\varepsilon > 0$ there exists a $\psi_R \in \mathcal{C}^\infty_0(\mathbb{R}^n)$  with $\psi_R(x) = 1$ for $|x|<R$, satisfying
\begin{align*}
\int_{\Omega^{\mathrm{ext}}} | \nabla w - \nabla (w \,  \psi_R )|^p  \,  \mathrm{d}x < \varepsilon.
\end{align*}
Let $u_1 \in E^{1,p}(\Omega^{\mathrm{ext}})$ be the first eigenfunction of \eqref{eq:PDEStek} with  $||u_1||_{L^p(\partial \Omega)} = 1$.
By choosing $R$ large enough, such that $\psi_R(x) = 1 $ for all $x \in \partial \Omega$, we have
\begin{align*}
\int_{\partial \Omega} |u_1 \, \psi_R |^p\,  \mathrm{d}S = \int_{\partial \Omega} |u_1|^p\,  \mathrm{d}S = 1.
\end{align*}
For  $\alpha < -\mu_1(p,n,\Omega^\mathrm{ext})$, we pick $R$ large enough such that $\phi := u_1 \,  \psi_R $ satisfies
\begin{align*}
\int_{\Omega^{\mathrm{ext}}} | \nabla \phi  |^p - |\nabla u_1|^p \,  \mathrm{d}x  <  \frac{-\mu_1(p,n,\Omega^\mathrm{ext}) - \alpha}{2},
\end{align*}
which implies
\begin{align*}
\int_{\Omega^{\mathrm{ext}}} | \nabla \phi  |^p \,  \mathrm{d}x + \alpha \int_{ \partial \Omega} | \phi  |^p\,  \mathrm{d}S &< \frac{-\mu_1(p,n,\Omega^\mathrm{ext}) - \alpha}{2} +\int_{\Omega^{\mathrm{ext}}} | \nabla u_1  |^p \,  \mathrm{d}x + \alpha \\
&= \frac{-\mu_1(p,n,\Omega^\mathrm{ext}) - \alpha}{2} +\mu_1(p,n,\Omega^\mathrm{ext}) + \alpha < 0.
\end{align*}
Using $\phi$ as a trial function in \eqref{eq:lambda1int}, yields $\lambda_1(\alpha,p,n,\Omega^{\mathrm{ext}}) < 0$.
\end{proof}



\begin{example}
Consider $\Omega = B_R \subset \mathbb{R}^n$, with $n \geq 2$ and $p \in (1,n)$. Then, the first eigenfunction of \eqref{eq:PDEStek} is  given by $u(x) = |x|^{-\frac{n-p}{p-1}}  $. The boundary condition of \eqref{eq:PDEStek} implies $ \mu_1(p,n, B_R^\mathrm{ext}) = \left( \frac{1}{R} \frac{n-p}{p-1} \right)^{p-1}$. Thus, according to Lemma \ref{lemma:alpha*}, 
\begin{align*}
\lambda_1(\alpha,p,n,B_R^\mathrm{ext}) < 0 \quad \Leftrightarrow \quad \alpha <  -\left(  \frac{n-p}{R(p-1)} \right)^{p-1}.
\end{align*}
\end{example}

\subsubsection{Supersolution Characterization of the First Eigenvalue}\label{subsub_super}

The following is an equivalent definition of the first variational eigenvalue, which is needed in the proof of Theorem~\ref{theo:cuteidea} in Section~\ref{sec:ball}. We define the set
\begin{align*}
S=&\{\lambda \in \R : \exists \, u >0 \text{ on } \Omega^{\text{ext}}, u \in W^{1, p}(\Omega^{\text {ext }}) \text{ with } -|\nabla u|^{p-2}\partial_\nu u +\alpha|u|^{p-2} u=0 \text { on } \partial \Omega\\ & \text{ and } \Delta_p u+\lambda|u|^{p-2} u \geq 0  \text { in } \Omega^{\text {ext }} \text{ (in the weak sense)} \}.
\end{align*}
And, if $S$ is not empty, we define $\lambda_1^* = \inf(S)$.

\begin{lemma}
If $\alpha < \alpha^*(p,n, \Omega^{\mathrm{ext}}) \leq 0$, then
$$
\lambda_1(\alpha, p,n, \Omega^{\mathrm{ext}}) = \lambda_1^*.
$$
\end{lemma}
\begin{proof}
By Lemma \ref{theo:existencemin} and the remarks afterwards, $\lambda_1(\alpha, p,n, \Omega^{\mathrm{ext}}) \in S$ for $\alpha < \alpha^*(p,n, \Omega^{\mathrm{ext}})$. Hence $S$ is nonempty and $\lambda_1^* \leq  \lambda_1(\alpha, p, n, \Omega^{\mathrm{ext}})$. On the other hand, for any $\lambda \in S$ there is a positive $u$ with
$$
\Delta_p u+\lambda|u|^{p-2} u \geq 0  \text { in } \Omega^{\mathrm{ext}},
$$
then
$$
\int_{\Omega^{\text {ext }}} (\Delta_p u+\lambda|u|^{p-2} u)  u \, \mathrm{d} x \geq 0
$$
and the divergence theorem implies
$$
-\int_{\Omega^{\text {ext }}}|\nabla u|^{p-2}\langle\nabla u, \nabla u \rangle \, \mathrm{d} x-\alpha \int_{\partial \Omega}|u|^{p-2} u^2 \, \mathrm{d} S+\lambda \int_{\Omega^{\text {ext }}}|u|^{p-2} u^2 \, \mathrm{d} x \geq 0,
$$
so that $\lambda \geq \lambda_1(\alpha, p, n,\Omega^\mathrm{ext}) $. In conclusion $\lambda_1^* =\lambda_1(\alpha, p, n,\Omega^\mathrm{ext})$.
\end{proof}

We end this subsection by providing a simple but useful scaling equality.

\begin{remark}\label{remark:scaling}
For $\beta >0$, a straightforward scaling argument leads to
\begin{align*}
\lambda_1( \alpha,  p, n, \beta \Omega^\mathrm{ext}) = \frac{\lambda_1( \beta^{p-1} \alpha,  p, n,  \Omega^\mathrm{ext})}{\beta^{p}} , \quad \quad \mu_1(  p, \Omega^\mathrm{ext}) = \beta^{p-1} \mu_1(  p, \beta \Omega^\mathrm{ext}).
\end{align*}
\end{remark}

\subsection{The Eigenvalue Problem on the Exterior of a Ball}\label{sec:ball}

When considering $\Omega = B_R \subset \mathbb{R}^n$, it is often possible to find explicit solutions of differential equations. For $n=1$, we do have an explicit form of the solution of \eqref{eq:PDE}, as discussed in the following remark. Here $\lambda_1(\alpha,p,1,B_R^\mathrm{ext})$ is independent of $R$.
 
\begin{remark}\label{remark:n=1}
Let $n=1$, $\alpha <0$, and $\Omega = (-R,R) $. If $u(x) = \phi(|x|) $ is the first nonnegative eigenfunction, the differential equation \eqref{eq:PDE}  reduces to
\begin{align*}
(p-1) \phi''(r) \left( -\phi'(r) \right)^{p-2} + \lambda_1(\alpha,p,1,B_R^\mathrm{ext}) \phi(r)^{p-1} &= 0 \quad \text{ for } r \in (R, \infty), \\
- \phi'(R) \left(- \phi'(R)\right)^{p-2} + \alpha \phi(R)^{p-1} &= 0,
\end{align*}
which has the solutions
\begin{align*}
\phi(r) = c \exp\left( - |\alpha|^\frac{1}{p-1} r \right), \quad \lambda_1(\alpha,p,1,\mathbb{R} \setminus [-R,R]) = -(p-1) |\alpha|^\frac{p}{p-1}.
\end{align*}
\end{remark}


However, for $p \neq 2$, $n \geq 2$, there exists no simple form for the solution of \eqref{eq:PDE}. As a first result, we show a monotonicity result with respect to the dimension.
\begin{lemma}\label{theo:ineq}
For $\Omega_n = B_R \subset \mathbb{R}^n$, $\Omega_{n-1} = B_R \subset \mathbb{R}^{n-1}$, $n \geq 2$,  $p \in (1, \infty)$ and $\alpha < \alpha^*(p,n,\Omega_n^\mathrm{ext})$, 
\begin{align*}
\lambda_1(\alpha,p,n,\Omega_n^\mathrm{ext}) > \lambda_1(\alpha,p,n-1,\Omega_{n-1}^\mathrm{ext}) \geq  -(p-1) |\alpha|^\frac{p}{p-1}.
\end{align*}
\end{lemma}
\begin{proof}
In view of Remark \ref{remark:scaling}, it is sufficient to show Lemma \ref{theo:ineq} for $R=1$. Since the eigenfunctions $u_1$ and $u_2$, corresponding to $\lambda_1(\alpha,p,n,\Omega_n^\mathrm{ext})$ and $\lambda_1(\alpha,p,n-1,\Omega_{n-1}^\mathrm{ext})$ respectively, are radial,
\begin{align*}
\lambda_1(\alpha,p,n,\Omega_n^\mathrm{ext}) &> \frac{\int_{1}^\infty |u_1'(r)|^p r^{n-2} \, \mathrm{d}r + \alpha |u_1(0)|^p}{\int_{1}^\infty |u_1(r)|^p r^{n-2} \, \mathrm{d}r} \geq  \lambda_1(\alpha,p,n-1,\Omega_{n-1}^\mathrm{ext}),
\end{align*}
where we used $\int_{1}^\infty |u_1'(r)|^p r^{n-2} \, \mathrm{d}r + \alpha |u_1(0)|^p < 0$. Remark \ref{remark:n=1} then implies the result.
\end{proof}

For small $\alpha$, we refine the lower bound and describe the behavior of $\lambda_1$ as $\alpha \to 0$, as shown in Theorem \ref{theo:asym}. 

\refstepcounter{theorem} 
\label{theo:asym}

\begin{proof}[Proof of Theorem \ref{theo:asym}]
By Remark~\ref{remark:scaling}, it is sufficient to consider $R=1$. To establish the upper bound, we define $u(x) := \exp(- \beta(\alpha) |x|)$, with $\beta(\alpha) := \left( \frac{|\alpha| p^n}{2\Gamma(n)} \right)^\frac{1}{p-n}$. Then,
\begin{align*}
\int_{B_1^\mathrm{ext}} |u|^p \, \mathrm{d}x &= \frac{|\partial B_1|}{(p \beta(\alpha))^n} \int_{p \beta(\alpha)}^\infty \exp(-r) r^{n-1} \, \mathrm{d}r \leq  \frac{|\partial B_1|}{(p \beta(\alpha))^n} \Gamma(n).
\end{align*}
Thus, for $\exp(-p \beta(\alpha)) > \frac{1}{2}$, we have
\begin{align*}
\lambda_1( \alpha,p,n,B_1^\mathrm{ext}) 
\leq  |\alpha|^\frac{p}{p-n} \frac{\frac{1}{2} - \exp(-p \beta(\alpha))}{\frac{\Gamma(n)}{p^n}  \left( \frac{ p^n}{2\Gamma(n)} \right)^\frac{-n}{p-n} }.
\end{align*}
Since $\lim_{ \alpha \nearrow 0} \exp(-p\beta(\alpha)) = 1$, this yields the desired upper bound.

For the lower bound, since $\alpha <0$ and  $p>n$ imply $\lambda_1(\alpha,p,n,B_1^\mathrm{ext}) <0$, we define $A_{R_0} := B_{R_0} \setminus \overline{B_1}$ and obtain for any $R_0>1$,
\begin{align*}
\lambda_1(\alpha,p,n,B_1^\mathrm{ext}) &\geq \inf_{w \in W^{1,p}(A_{R_0})} \frac{\int_{A_{R_0}} | \nabla w|^{p}  \mathrm{d}x + \alpha \int_{\partial B_1} |w|^p \mathrm{d}S}{\int_{A_{R_0}} |w|^p \mathrm{d}x} =: \Lambda_1(\alpha,p,R_0).
\end{align*}
Given the smoothness of the ball, standard regularity theory for elliptic equations implies that a minimizer $w \in W^{1,p}(A_{R_0})$ of $\Lambda_1(\alpha,p,R_0)$ is a  solution of
\begin{align*}
\begin{cases}
\Delta_p w + \Lambda_1(\alpha,p, R_0) |w|^{p-2} w = 0 \, &\text{ in } A_{R_0}, \\
-  | \nabla w|^{p-2} \partial_\nu w + \alpha |w|^{p-2} w = 0 \, &\text{ on } \partial B_1,\\
\partial_\nu w  = 0 \, &\text{ on } \partial B_{R_0}.
\end{cases}
\end{align*}
Since $w$ is radial, we write $w(x) = f(|x|)$ and  choose $w$ such that $f(r) \geq 0$ and $\int_{1}^{R_0} |f(r)|^p r^{n-1} \, \mathrm{d}r = 1$. Then, $f$ is decreasing and
\begin{align}\label{eq:Lambda1rad}
\Lambda_1(\alpha,p, R_0) =  \int_{1}^{R_0}  | f'(r)|^p r^{n-1} \, \mathrm{d}r + \alpha f(1)^p.
\end{align}
To find a lower bound of $\lambda_1(\alpha, p,n, B_1^\mathrm{ext})$, we  derive an upper bound for $f(1)^p$ that depends on $R_0$ and then study its behavior as $R_0 \to \infty$. Since $f'(r) < 0$, 
\begin{align*}
f(1)^p &= \left( f(1)^\frac{p}{n} \right)^n = \left( f(R_0)^\frac{p}{n}- \int_1^{R_0}  \left( f(r)^\frac{p}{n}  \right)' \, \mathrm{d}r \right)^n
=\left( f(R_0)^\frac{p}{n}+ \int_1^{R_0}   \frac{p}{n} |f'(r)| f(r)^{\frac{p-n}{n}}  \, \mathrm{d}r \right)^n,
\end{align*}
and
\begin{align}\label{eq:uneqf(R_0)}
1 = \int_1^{R_0} f(r)^p r^{n-1} \, \mathrm{d}r > \int_1^{R_0} f(R_0)^p r^{n-1} \, \mathrm{d}r = f(R_0)^p \,\frac{R_0^n-1}{n}.
\end{align}
So, $ \lim_{R_0 \to \infty} f(R_0)= 0$. Furthermore, H\"older's inequality yields
\begin{align*}
&\int_1^{R_0} |f'(r)| f(r)^{\frac{p-n}{n}}  \, \mathrm{d}r = \int_1^{R_0}  |f'(r)| r^\frac{n-1}{p} f(r)^{\frac{p-n}{n}} r^{-\frac{n-1}{p}}  \, \mathrm{d}r \\
\leq& \left( \int_1^{R_0}  |f'(r)|^p r^{n-1}  \, \mathrm{d}r  \right)^\frac{1}{p} \left(  \int_1^{R_0} f(r)^{\frac{p-n}{n}\frac{p}{p-1}} r^{\frac{-n+1}{p}\frac{p}{p-1}}  \, \mathrm{d}r \right)^\frac{p-1}{p}.
\end{align*}
The second integral can be estimated, using again H\"older's inequality as well as the normalization $\int_{1}^{R_0}  | f(r)|^p r^{n-1} \, \mathrm{d}r = 1$, as follows
\begin{align*}
\left(  \int_1^{R_0} f(r)^{\frac{p-n}{n}\frac{p}{p-1}} r^{\frac{-n+1}{p}\frac{p}{p-1}}  \, \mathrm{d}r \right)^\frac{p-1}{p} 
\leq  \left( \int_1^{R_0}   r^{-1}   \, \mathrm{d}r \right)^\frac{n-1}{n} = \ln(R_0)^\frac{n-1}{n}.
\end{align*}
Setting $y := \left( \int_1^{R_0}  |f'(r)|^p r^{n-1}  \, \mathrm{d}r  \right)^\frac{1}{p}$ and $c(R_0) :=   \frac{2^{n-1} p^n}{n^n}  \ln(R_0)^{n-1}$, we have
\begin{align*}
\Lambda_1(\alpha,p, R_0) &\geq   y^p + \alpha \left( f(R_0)^\frac{p}{n}+ \frac{p}{n} y     \ln(R_0)^\frac{n-1}{n} \right)^n 
\geq y^p -|\alpha| c(R_0) y^n - |\alpha | 2^{n-1} f(R_0)^p.
\end{align*}
Since  $y \mapsto y^p - k y^n$ attains its minimum at $ \left( \frac{kn}{p} \right)^\frac{1}{p-n}$, we see that
\begin{align*}
\Lambda_1(\alpha,p, R_0) &\geq \left( \frac{c(R_0)|\alpha|n}{p} \right)^\frac{p}{p-n} -|\alpha|c(R_0) \left( \frac{c(R_0)|\alpha|n}{p} \right)^\frac{n}{p-n} - \frac{|\alpha| 2^{n}}{2} f(R_0)^p\\
&= |\alpha|^\frac{p}{p-n} c(R_0)^\frac{p}{p-n} \left[ \left( \frac{n}{p} \right)^\frac{p}{p-n} -  \left( \frac{n}{p} \right)^\frac{n}{p-n} \right] - \frac{|\alpha| 2^{n}}{2} f(R_0)^p.
\end{align*}
From \eqref{eq:uneqf(R_0)}, we know $f(R_0)^p = \mathcal{O}(R_0^{-n})$, and by choosing $R_0 := \frac{1}{|\alpha|^{p-n}}$, we obtain, for $\alpha \nearrow 0$, the asymptotic behavior
\begin{align*}
\frac{\lambda_1( \alpha,p,n,B_1^\mathrm{ext})}{|\alpha|^{\frac{p}{p-n}-\varepsilon}} 
 =  \mathcal{O}\left( |\alpha|^\varepsilon \ln(|\alpha|)^{\frac{p(n-1)}{p-n}} \right) + \mathcal{O}\left( |\alpha|^\varepsilon \ln(|\alpha|)^{\frac{n^2-2n+p}{p-n}} \right) + \mathcal{O} \left( |\alpha|^{\varepsilon} \right)
\end{align*}
which proves the statement.
\end{proof}

By Remark~\ref{remark:n=1},  $\lambda_1(\alpha,p,1,B_R^\mathrm{ext})$ is independent of $R$. Consequently, $\lim_{R \to 0}\lambda_1(\alpha,p,1,B_R^\mathrm{ext}) \neq 0$. Given that the capacity of a point is strictly positive for $p>n$, see e.g. \cite[Example 2.12]{heinonen2018nonlinear}, one might be led to conjecture that $\lim_{R \to 0}\lambda_1(\alpha,p,n,B_R^\mathrm{ext}) \neq 0$ for $n \geq 2$ as well. However, as a consequence of Theorem \ref{theo:asym}, we conclude that this is not the case.

\begin{korollar}\label{coro:cap}
Suppose $\alpha < 0$. For $2 \leq n < p$, it holds $\lim_{R \to 0} \lambda_1(\alpha,p,n,B_R^\mathrm{ext}) =  0$.
\end{korollar}
\begin{proof}
Using Remark \ref{remark:scaling}, we notice that
\begin{align*}
\lambda_1(\alpha,p,n,B_R^\mathrm{ext})  = \frac{|R^{p-1} \alpha |^{\frac{p}{p-n}-\varepsilon}}{R^p} \frac{\lambda_1( R^{p-1}\alpha,p,n,B_1^\mathrm{ext})}{|R^{p-1} \alpha |^{\frac{p}{p-n}-\varepsilon}},
\end{align*}
for any $\varepsilon > 0$. By Theorem \ref{theo:asym}, both factors vanish as $R \to 0$  if $\varepsilon < \frac{p(n-1)}{p-n}$.
\end{proof}

For bounded domains, $\alpha >0$ and $p \in (1,\infty)$, it is known that the mapping $R \mapsto \lambda_1(\alpha, p,n, B_R)$ is strictly monotonically decreasing. This result is independently shown by Bucur and Daners in \cite[Lemma 4.1]{bucur2010alternative}  and by Dai and Fu in \cite[Proposition 2.8]{dai2011faber}. To prove the monotonicity of $R \mapsto \lambda_1(\alpha, p,n, B_R^\mathrm{ext})$, we need the following lemma.

\begin{lemma}\label{lemma:dai}
Suppose $\Omega = B_R \subset \mathbb{R}^n$, $n \geq 2$, and $p \in (1, \infty)$.  For $\alpha < \alpha^*( p, n, B_R^\mathrm{ext})$, let $u: B_R^\mathrm{ext} \to \mathbb{R}$ denote the non-negative eigenfunction corresponding to $\lambda_1(\alpha, p,n, B_R^\mathrm{ext})$. As $u$ is radial, we can write $u(x) = \phi(|x|)$, where $\phi:[R, \infty) \to \mathbb{R}$. Then, $\phi$ is strictly logarithmically concave and satisfies
\begin{align*}
\frac{\phi'(R)}{\phi(R)} = -|\alpha|^\frac{1}{p-1} \quad \text{and } \quad \lim_{r \to \infty} \frac{\phi'(r)}{\phi(r)} = -\left( \frac{-\lambda_1(\alpha,p,n,B_R^\mathrm{ext})}{p-1} \right)^\frac{1}{p}. 
\end{align*} 
\end{lemma}
\begin{proof}
Since $\phi$ is strictly positive, 
\begin{align*}
\frac{\phi(r)}{\phi(R)} = \exp\left( \ln(\phi(r)) - \ln(\phi(R)) \right)  = \exp\left( \int_{R}^r \frac{\phi'(t)}{\phi(t)} \, \mathrm{d}t \right).
\end{align*}
We define $g(r) := \frac{-\phi'(r)}{\phi(r)}=-\frac{\mathrm{d}}{\mathrm{d}r} \ln( \phi(r))$. Then, $g(r) >0$ as $\phi'(r) < 0$, and 
\begin{align}\label{eq:defg}
\phi(r) = \phi(R) \exp\left( -\int_{R}^r g(t) \, \mathrm{d}t \right).
\end{align} 
First, we establish the existence of $\lim_{r \to \infty} g(r)$. Using this property, we infer the monotonicity of $g$, which is equivalent to the logarithmic concavity. Differentiating \eqref{eq:defg} with respect to $r$ yields
\begin{align*}
\phi'(r) = -g(r) \phi(r) \quad \text{ and } \quad \phi''(r) = -g'(r) \phi(r) +g(r)^2 \phi(r).
\end{align*}
In addition, equation \eqref{eq:PDE} implies that  $\phi$ solves
\begin{align*}
(p-1) \phi''(r) ( - \phi'(r))^{p-2} - \frac{n-1}{r} (-\phi'(r))^{p-1} + \lambda_1 \phi(r)^{p-1} = 0,
\end{align*}
where  $\lambda_1 = \lambda_1(\alpha, p,n, B_R^\mathrm{ext})$.  Dividing this equation by $( - \phi'(r))^{p-2}$, leads to
\begin{align}\label{eq:lemma1}
(p-1) \phi''(r)  - \frac{n-1}{r} (-\phi'(r)) + \lambda_1 \frac{\phi(r)}{g^{p-2}(r)}  = 0.
\end{align}
Rearranging \eqref{eq:lemma1} and using $\phi''(r) = -g'(r) \phi(r) +g(r)^2 \phi(r)$, we obtain
\begin{align}
\lambda_1   
&= (p-1) g^{p-2}(r) g'(r) - (p-1) g^{p}(r)   + \frac{n-1}{r} g^{p-1}(r). \label{eq:lemma10rear}
\end{align}
Differentiating this equality with respect to $r$ and dividing by $g^{p-2}(r)(p-1)$, we obtain
\begin{align*}
 g''(r) &= -(p-2)  \frac{g'(r)^2}{g(r)}  +p  g(r) g'(r)   + \frac{n-1}{(p-1)r^2} g(r)  - \frac{n-1}{r}  g'(r).
\end{align*}
Therefore, if there exists a point $r_0 \in (R, \infty)$ such that $g'(r_0) = 0$, then  
\begin{align*}
g''(r_0) = \frac{n-1}{(p-1)r_0^2} g(r_0) >0.
\end{align*}
Consequently, $g$ has a strict local minimum at $r_0$, which excludes the existence of any other critical points. This means $g$ is either monotonically decreasing on $[R, \infty)$ (if no critical point exists) or monotonically increasing on $[r_0, \infty)$. Therefore, we have either $\lim_{r \to \infty} g(r) = \infty$ or $\lim_{r \to \infty} g(r) = c \geq 0$. Given that both $\phi$ and $\phi'$ belong to $L^p([R,\infty))$ and are monotonic, we can apply  L'H\^{o}pital's rule, and using \eqref{eq:lemma1}, we see that
\begin{align*}
\lim_{r \to \infty} g(r) = \lim_{r \to \infty} \frac{- \phi''(r)}{\phi'(r)}  = \lim_{r \to \infty} \frac{\frac{n-1}{r} - \frac{\lambda_1}{g^{p-1}(r)}}{p-1}.
\end{align*}
Hence, neither $\lim_{r \to \infty} g(r) = \infty$ nor $\lim_{r \to \infty} g(r) = 0$ can occur, leaving $c = \left({\frac{- \lambda_1}{p-1}}\right)^{1/p}$. 
The boundary condition on $\partial B_R$ implies $g(R)^{p-1} = - \alpha$, which is equivalent to  $g(R) = \sqrt[p-1]{- \alpha}$. Given that $g$ has no local maxima, we conclude
\begin{align*}
g(r) \leq \max \left\{ \sqrt[p]{\frac{- \lambda_1}{p-1}} , \sqrt[p-1]{- \alpha} \right\} = \sqrt[p-1]{- \alpha},
\end{align*}
where we used Lemma \ref{theo:ineq} for the second inequality. Knowing this limit, allows us to dismiss the existence of a critical point: Suppose there existed a critical point $r_0$, then $g(r_0) \leq \lim_{r \to \infty} g(r)$. From \eqref{eq:lemma10rear}, we infer
\begin{align*}
\lambda_1 &=  - (p-1) g^{p}(r_0)   + \frac{n-1}{r_0} g^{p-1}(r_0)  >  - (p-1) \left( \lim_{r \to \infty} g(r) \right)^p  = \lambda_1.
\end{align*}
Thus, there exists no critical point, and $g$ must be strictly monotonically decreasing. Furthermore, $g(r) = - \frac{\mathrm{d}}{\mathrm{d}r} \ln(\phi(r)) $. Thus, $
\frac{\mathrm{d^2}}{\mathrm{d}r^2} \ln(\phi(r)) = -g'(r) > 0$ and $\phi$ is strictly logarithmically concave.
\end{proof}

For $p=n=2$, the first eigenfunction is explicitly known, allowing to describe the asymptotic behavior of $\lambda_{1}(\alpha,2,2,B_1^{\mathrm{ext}})$. Lemma \ref{lemma:dai} and the supersolution characterization introduced in Section \ref{subsub_super} allow us to extend this to $n \geq 2$. In Theorem \ref{theo:cuteidea} we state this result for $B_1^{\mathrm{ext}}$, but note that thanks to Remark \ref{remark:scaling}, analogous bounds hold for general $B_R^{\mathrm{ext}}$.

\refstepcounter{theorem} 
\label{theo:cuteidea}

\begin{proof}[Proof of Theorem \ref{theo:cuteidea}]
For $\alpha< \alpha^*(p,n, B_1^{\text{ext }})$, as in Lemma \ref{lemma:dai}, $\lambda_1 = \lambda_1\left(\alpha, p, n, B_1^{\text{ext }}\right)<0$ has a radial eigenfunction $\varphi(r) > 0$,  $\varphi^{\prime}(r)<0$, which solves
$$
\left\{\begin{array}{l}
\varphi^{\prime \prime}(r)+\frac{n-1}{p-1} \frac{1}{r} \varphi^{\prime}(r)+\frac{\lambda_1}{p-1}\left(\frac{\varphi(r)}{-\varphi^{\prime}(r)}\right)^{p-2} \varphi(r)=0 \quad \text{ for } r \in (1,\infty), \\
\varphi^{\prime}(1)=-|\alpha|^{\frac{1}{p-1}} \varphi(1).
\end{array}\right.
$$
For $n=p$, 
the differential equation becomes
$$\varphi_n^{\prime \prime}(r)+\frac{1}{r} \varphi_n^{\prime}(r)+\frac{\lambda_{1}}{n-1}\left(\frac{\varphi_n(r)}{-\varphi_n^{\prime}(r)}\right)^{n-2} \varphi_n(r)=0.$$
\noindent
If we can find a $ k<0 $ such that
$$
\begin{aligned}
 0 &=\varphi_n^{\prime \prime}(r)+\frac{1}{r} \varphi_n^{\prime}(r)+\frac{\lambda_1}{n-1}\left(\frac{\varphi_n(r)}{-\varphi_n^{\prime}(r)}\right)^{n-2} \varphi_n(r) \leq \varphi_n^{\prime \prime}(r)+\frac{1}{r} \varphi_n^{\prime}(r)+k \varphi_n(r), 
\end{aligned}
$$
the supersolution characterization of Section \ref{subsub_super},  with $n=p=2$, would give
$$
0 >k \geq \lambda_{1}(-|\alpha|^{\frac{1}{n-1}},2,2,B_1^{\mathrm{ext}})  
.$$
Note that by cancellation and the positivity of $\varphi_n(r)$, $k$ needs only to satisfy

$$0 < -k \leq \frac{-\lambda_{1}(\alpha,n,n,B_1^{\mathrm{ext}})}{(n-1)}\left(\frac{\varphi_n(r)}{-\varphi_n^{\prime}(r)}\right)^{n - 2}.$$
\noindent
By Lemma~\ref{lemma:dai}, the last expression inside the parentheses of the last factor is minimized at $r=1$ and we can pick $k$ accordingly, yielding
$$\lambda_{1}(-|\alpha|^{\frac{1}{n-1}},2,2,B_1^{\mathrm{ext}}) \leq k=\frac{\lambda_{1}(\alpha,n,n,B_1^{\mathrm{ext}})}{n-1}\left(\frac{1}{|\alpha|}\right)^{\frac{n-2}{n-1}}<0.$$
Therefore, we have
\begin{equation}\label{cuteqn1}
0>\lambda_{1}(\alpha,n,n,B_1^{\mathrm{ext}}) \geq(n-1)|\alpha|^{\frac{n-2}{n-1}} \lambda_{1}(-|\alpha|^{\frac{1}{n-1}},2,2,B_1^{\mathrm{ext}}).
\end{equation}
We can reverse the argument above. For $\alpha < 0$ there is a negative eigenvalue $\lambda_1 = \lambda_{1}(\alpha,2,2,B_1^{\mathrm{ext}})$ and a positive eigenfunction $\varphi_2(r)$ satisfying
$$
\begin{aligned}
& \left\{\begin{array}{l}
\varphi_2^{\prime \prime}+\frac{1}{r} \varphi_2^{\prime}(r)+\lambda_{1} \varphi_2(r) = 0 \quad \text{ for } r \in (1,\infty)  \\
\varphi_2^{\prime}(1)=\alpha \varphi_2(1).
\end{array}\right. 
\end{aligned}
$$
If there exists $ k<0 $ with   
$$
 0=\varphi_2^{\prime \prime}(r)+\frac{1}{r} \varphi_2^{\prime}(r)+\lambda_{1} \varphi_2(r)
\leq \varphi_2^{\prime \prime}(r)+\frac{1}{r} \varphi_2^{\prime}(r)+\frac{k}{n-1}\left(\frac{\varphi_2(r)}{-\varphi_2(r)}\right)^{n-2} \varphi_2(r),
$$
then the supersolution characterization, now with $n=p$ guarantees

$$
0>k \geq  \lambda_{1}(-|\alpha|^{n-1},n,n,B_1^{\mathrm{ext}}).
$$
Again using Lemma~\ref{lemma:dai}, but minimizing as $r \rightarrow \infty$, one can set
$$
\begin{aligned}
-k & =-\lambda_{1}(n-1)  \inf_{r \in [1,\infty)} \left(g_2(r)^{n-2}\right)=-\lambda_{1} (n-1) \left(\frac{-\lambda_{1}}{2-1}\right)^{\frac{n-2}{2}}.  
\end{aligned}
$$
So that 
$$
0> -(n-1) (-\lambda_{1}(\alpha,2,2,B_1^{\mathrm{ext}}))^\frac{n}{2}\geq  \lambda_{1}(-|\alpha|^{n-1},n,n,B_1^{\mathrm{ext}}),
$$ and rescaling gives
$$
0> -(n-1) (-\lambda_{1}(-|\alpha|^{\frac{1}{n-1}},2,2,B_1^{\mathrm{ext}}))^\frac{n}{2}\geq  \lambda_{1}(\alpha,n,n,B_1^{\mathrm{ext}}).
$$ 
Combining this inequality with (\ref{cuteqn1}), we obtain our claim.
\end{proof}

Lemma \ref{lemma:dai}, allows us to follow the proof of \cite[Theorem 1]{giorgi2005monotonicity}, yielding the corollary below. In particular, Corollary \ref{theo:tiziana} implies the monotonicity of $R \mapsto \lambda_1(\alpha,p,n,B_R^\mathrm{ext})$.  Monotonicity with respect to domain inclusion does not hold in general, even for convex domains, \cite[Remark 1]{bundrock2024optimizing}.

\begin{corollary}\label{theo:tiziana}
Let $ p \in (1, \infty)$, $n \geq 2$, and let $\Omega \subset \mathbb{R}^n$ be a Lipschitz domain with $B_r \subseteq \Omega$. For $n > p$ and  $\alpha < \alpha^*(p,n, B_r^\mathrm{ext})$ or $n \leq p$ and $\alpha <0$, it holds
\begin{align*}
\lambda_1(\alpha, p,n, \Omega^{\mathrm{ext}}) \leq \lambda_1(\alpha, p,n,B_r^{\mathrm{ext}}).
\end{align*}
\end{corollary}

\subsection{Geometric Inequalities}\label{sec:pacwoman}

For $\alpha <0$, Krej{\v{c}}i{\v{r}}{\'\i}k and Lotoreichik prove that the exterior of a ball maximizes $\lambda_1(\alpha, 2,2,\Omega^\mathrm{ext})$ among all smooth bounded domains $\Omega \subset \mathbb{R}^2$ with given perimeter or given area, \cite{krejcirik2016optimisation,krejvcivrik2020optimisation}.  In higher dimensions, the exterior of a ball no longer maximizes $\lambda_1(\alpha, 2,n,\Omega^\mathrm{ext})$, \cite[Section 5.3]{krejcirik2016optimisation}, \cite[Section 3.5]{bundrock2023robin}, \cite[Section 2.4]{bundrock2024optimizing}.
%
%
%
%
%
%
Based on the results from Section \ref{sec:existence} and \ref{sec:ball}, we can apply the same arguments as in \cite{krejvcivrik2020optimisation}, and obtain the following result for the $p$-Laplacian in two dimensions.
 
 \begin{theorem}\label{theo:5}
For $n=2$, $p \in (2, \infty)$, $\alpha < 0$ or $p \in (1,2)$, $\alpha < \alpha^*(p, 2,B_R^\mathrm{ext})$,
\begin{align*}
\max \lambda_1(\alpha, p,2, \Omega^\mathrm{ext}) = \lambda_1(\alpha, p,2, B_R^\mathrm{ext}) ,
\end{align*}
where the maximum is taken over all smooth bounded, simply connected, open sets $\Omega \subset \mathbb{R}^2$,  such that $|\Omega| = | B_R|$ or $|\partial \Omega| = |\partial B_R|$.
\end{theorem}


To show that this cannot be extended to higher dimensions, we consider the following example.
\begin{example}\label{example:2}
Suppose $n \geq 3$ and $p \in (1, \infty)$. For $a \in (0,1)$, we define
\begin{align*}
E(a):= \left\{ x \in \mathbb{R}^n: (a x_1)^2 + \sum_{k=2}^n x_k^2 <1 \right\}.
\end{align*}
Then, for sufficiently small $a$, we have
\begin{align*}
 H_\text{max}(E(a)^\mathrm{ext}) &= - \frac{n-2+a^2}{n-1} < 
 H_\text{max}\left(B^\mathrm{ext}\right),
\end{align*}
where $H(\Omega^\mathrm{ext}) = -H(\Omega) $ and $H_\text{max}(\Omega^\mathrm{ext})$ is the maximal curvature of $\partial \Omega^\mathrm{ext}$ and $B$ is the ball with $|B| = |E(a)|$. Kova{\v{r}}{\'\i}k and Pankrashkin show the asymptotic behavior of $\lambda_1$ as $\alpha \to - \infty$, \cite[Theorem 1.1]{kovavrik2017p}. For our problem, this reads as
\begin{align}\label{eq:asym}
\lambda_1(\alpha,p,n,\Omega^\mathrm{ext}) = -(p-1) | \alpha|^\frac{p}{p-1} - (n-1) H_\text{max} ( \Omega^\mathrm{ext}) |\alpha| + o(\alpha).
\end{align}
Thus, for sufficiently negative $\alpha$, we have $\lambda_1(\alpha,p,n,E(a)^\mathrm{ext}) >\lambda_1 \left( \alpha,p,n,B^\mathrm{ext} \right)$.
\end{example}

\refstepcounter{theorem} 
\label{prop:bound}
 
Finally, we discuss lower bounds for $\lambda_1$ in Theorem \ref{prop:bound}.
\begin{proof}[Proof of Theorem \ref{prop:bound}]
For the first part, we construct a sequence $\Omega_m$ that tends to a non-Lipschitz domain, see Figure \ref{fig:pacman}. For $m \in \mathbb{N}$, we define
\begin{align*}
\Omega_m := B_1(0) \setminus \{ x \in \mathbb{R}^2: x_1 \geq m^{-1} \land  |x_2| \leq x_1^{p+3} \}.
\end{align*}
Let $u: \Omega_m^\mathrm{ext} \to \mathbb{R}$ be given by $u(x) := |x|^{-\frac{3}{p}}$. Then,
\begin{align*}
\int_{\Omega_m^\mathrm{ext}} |u|^p \, \mathrm{d}x \leq \frac{1}{p+1} + |\partial B_1| \qquad \text{ and } \qquad \int_{\Omega_m^\mathrm{ext}} |\nabla u|^p \, \mathrm{d}x &\leq 2\frac{3^p}{p^p} + \frac{3^p}{p^p} \frac{|\partial B_1|}{p+1}. 
\end{align*}
For the boundary integral, define  $\gamma_1: [{\frac{1}{m}},1] \to \mathbb{R}^2, \, t \mapsto \begin{pmatrix}
\frac{t}{\sqrt{2}} \\ \frac{t^{p+3}}{\sqrt{2}}
\end{pmatrix}$. Then,
\begin{align*}
\int_{\partial \Omega_m} u^2 \, \mathrm{d}S &>  2 \int_{{\frac{1}{m}}}^1 u^p(\gamma_1(t)) | \dot{\gamma_1}(t) | \, \mathrm{d} t  > 2 \int_{{\frac{1}{m}}}^1 \frac{1}{\sqrt{2} t^6}  \, \mathrm{d}t.
\end{align*}
Hence, $\lim_{m \to \infty} \int_{\partial \Omega_m} u^p \, \mathrm{d}S = \infty$. Thus, $\lim_{m \to \infty} \lambda_1(\alpha,p,2,\Omega_m^\mathrm{ext}) = - \infty$.

For the second part, note that any convex, bounded set has a Lipschitz boundary, \cite[Lemma 2.3]{dekel2004whitney}. Thus, $\lambda_1$ is well defined.  Since $\alpha < 0$, $p=n$, it holds $\lambda_1(\alpha,2,2,\Omega^\mathrm{ext}) < 0$. Thus, there exists an eigenfunction $u \in W^{1,2}(\Omega^\mathrm{ext})$, corresponding to $\lambda_1(\alpha,2,2,\Omega^\mathrm{ext})$.

As in \cite[Section 3]{kovavrik2017p}, we parameterize $\Omega^\mathrm{ext}$ using parallel coordinates,
\begin{align*}
\Phi: \partial \Omega \times (0, \infty) \to \Omega^\mathrm{ext}, \quad \Phi(s,t) := s + t \nu(s).
\end{align*}
We define $\phi(s,t) := 1 - \kappa(s) t$, where $-\kappa(s)$ is the  curvature of $\partial \Omega$. The convexity of $\Omega$ implies $\kappa(s) \leq 0$. Then, 
\begin{align*}
\int_{\Omega^\mathrm{ext}} u(x)^2 \, \mathrm{d}x = \int_{\partial \Omega \times (0, \infty)}  u \left( \Phi(s,t) \right)^2 \phi(s,t) \, \mathrm{d}S_s \, \mathrm{d}t.
\end{align*}
Additionally, we define $g := u \circ \Phi$, and have
\begin{align*}
 \left| (\nabla u)\circ \Phi(s,t) \right|^2   = \frac{|\partial_s g(s,t)|^2}{(1- \kappa(s) t)^2} + |\partial_t g(s,t) |^2.
\end{align*}
Therefore, $\lambda_1(\alpha,2, 2,\Omega^\mathrm{ext})$ equals
\begin{align*}
 &  \frac{\int_{\partial \Omega \times (0, \infty)} \left[ \frac{|\partial_s g(s,t)|^2}{(1- \kappa(s) t)^2} + |\partial_t g(s,t) |^2 \right] \phi(s,t) \, \mathrm{d}S_s \, \mathrm{d}t + \alpha \int_{\partial \Omega} | g(s,0) |^2 \, \mathrm{d}S_s }{\int_{\partial \Omega \times (0, \infty)} |g(s,t)|^2 (1- \kappa(s) t) \, \mathrm{d}S_s \, \mathrm{d}t} \\
 \geq&  \frac{\int_{\partial \Omega \times (0, \infty)}  |\partial_t g(s,t) |^2(1- \kappa(s) t) \, \mathrm{d}S_s \, \mathrm{d}t + \alpha \int_{\partial \Omega} | g(s,0) |^2 \, \mathrm{d}S_s }{\int_{\partial \Omega \times (0, \infty)} |g(s,t)|^2 (1- \kappa(s) t) \, \mathrm{d}S_s \, \mathrm{d}t}.
\end{align*}
To obtain a lower bound for $\lambda_1(\alpha,2,2, \Omega^\mathrm{ext})$, we define
\begin{align*}
\Lambda_1^\alpha(s) := \inf_{h \in W^{1,p}( (0, \infty))} \frac{\int_{0}^\infty |h'(t)|^2(1- \kappa(s) t) \, \mathrm{d}t + \alpha  | h(0) |^2  }{\int_{0}^\infty |h(t)|^2 (1- \kappa(s) t) \, \mathrm{d}t}.
\end{align*}
Any $h \in W^{1,p}( (0, \infty))$ satisfies
\begin{align*}
\int_{0}^\infty |h'(t)|^2(1- \kappa(s) t) \, \mathrm{d}t + \alpha  | h(s,0) |^2   \geq \Lambda_1^\alpha(s)\int_{0}^\infty |h(t)|^2 (1- \kappa(s) t) \, \mathrm{d}t,
\end{align*}
which implies
\begin{align*}
\lambda_1(\alpha,2, 2,\Omega^\mathrm{ext})  \geq \frac{\int_{\partial \Omega} \left[ \Lambda_1^\alpha(s)\int_{0}^\infty |g(s,t)|^2 (1- \kappa(s) t) \, \mathrm{d}t \right] \, \mathrm{d}S_s }{\int_{\partial \Omega} \int_0^\infty |g(s,t)| ^2 (1- \kappa(s) t) \, \mathrm{d}t \, \mathrm{d}S_s } \geq \inf_{s \in \partial \Omega} \Lambda_1^\alpha(s).
\end{align*}
For a given $s \in \partial \Omega$, we can see $\Lambda_1^\alpha(s)$  as the first eigenvalue of
\begin{align*}
\begin{cases}
\left( u'(t) (1- \kappa(s)t) \right)'  + \Lambda^\alpha(s) u(t) (1- \kappa(s)t) = 0 \quad &\text{ in } (0, \infty),\\
u'(0) - \alpha u(0) = 0 \quad \text{ and } \quad  \lim_{x \to \infty} u(x) =0.
\end{cases}
\end{align*}
If $\kappa(s) = 0$, we have $\Lambda_1^\alpha(s) = -\alpha^2$.
For $\kappa(s) < 0$, the solutions, which decay at infinity, are given by
\begin{align*}
u(t) = c K_0 \left( \frac{\sqrt{- \Lambda_1^\alpha(s)} (1- \kappa(s) t)}{- \kappa(s)} \right), \quad c \in \mathbb{R}.
\end{align*}
Hence, the boundary condition yields
\begin{align*}
\alpha = \frac{u'(0)}{u(0)}  = -\sqrt{-\Lambda_1^\alpha(s)} \frac{K_1 \left( \frac{\sqrt{- \Lambda_1^\alpha(s)} }{- \kappa(s)} \right)}{K_0 \left( \frac{\sqrt{- \Lambda_1^\alpha(s)} }{- \kappa(s)} \right)}.
\end{align*}
Using \cite[Theorem 5]{segura2011bounds}, this implies $ \Lambda_1^\alpha(s) \geq -\alpha^2$, which completes the proof.\end{proof}

\begin{appendix}
\section{Appendix}
\begin{lemma}\label{lemma:basicineq}
Suppose $a,b,c,d \in \mathbb{R}$, $a,b,c,d\geq 0$ and $c,d \neq 0$. Then,
\begin{enumerate}[(i)]
\item $2^{2-p} \frac{( a + b)^p}{c^p + d^p} \leq \frac{a^p}{c^p} + \frac{b^p}{d^p}$ for any $p \geq 2$.\label{eq:algineq}
\item $ \frac{( a + b)^p}{c^p + d^p} \leq \frac{a^p}{c^p} + \frac{b^p}{d^p}$ for any $1 \leq p \leq 2$.\label{eq:algineqimp}
\item $ ( a + b)^p \leq (1+\varepsilon)^{p-1}a^p + (1+\frac{1}{\varepsilon})^{p-1}b^p$ for any $p \geq 1$, $\varepsilon>0$. \label{eq:algineqcd1}
\item $(a+b)^p \leq b^p +  p \left( a^p + a b^{p-1} \right)$ for any $1 \leq p \leq 2$. \label{eq:ungl2}
\item $(a+b)^p \leq b^p + 2^{p-2} p \left( a^p + a b^{p-1} \right)$ for any $p \geq 2$. \label{eq:ungl1}
\end{enumerate}
\end{lemma}
\begin{proof}
The inequalities presented here can all be reduced to properties of a one-dimensional function.
\begin{enumerate}[(i)]
\item The claimed inequality is equivalent to
\begin{align}\label{eq:lemma11}
2^{2-p} (a+b)^p \leq a^p (1+x)+b^p\frac{x+1}{x},
\end{align}
where $x:=\frac{d^p}{c^p}\in (0, \infty)$. The right hand side attains its minimum at $x^* = \sqrt{\frac{b^p}{a^p}}$ and equals $\left( a^\frac{p}{2} + b^\frac{p}{2} \right)^2$. For a given $q \geq 1 $,  the function
\begin{align*}
f:[0, \infty) \to \mathbb{R}, \quad f(y) := 1+ y^q - 2^{1-q} \left( 1+y \right)^q
\end{align*}
is nonnegative. With $q=\frac{p}{2}$, $y = \frac{a}{b}$, this yields $0 < 1+\left(\frac{a}{b}\right)^\frac{p}{2}-2^{1-\frac{p}{2}} \left( 1+\frac{a}{b} \right)^\frac{p}{2}$, which implies \eqref{eq:lemma11}.
\item For a given $\frac{1}{2} \leq q<1$, the function $f:[0, \infty) \to \mathbb{R}$, $f(y) := 1+ y^q -  \left( 1+y \right)^q$ is nonnegative. As before, with $q=\frac{p}{2}$ and $y=\frac{a}{b}$, this yields (ii).
\item For $a \neq 0$, Lemma \ref{lemma:basicineq}\eqref{eq:algineqcd1} is equivalent to
\begin{align*}
 \left( 1 + \frac{b}{a} \right)^p \leq (1+\varepsilon)^{p-1} + \left( 1+\frac{1}{\varepsilon}\right)^{p-1}\left( \frac{b}{a} \right)^p.
\end{align*}
Using the convexity of $f:[0,\infty) \to \mathbb{R}$, $f(x)=x^p$ for $p>1$, we have 
\begin{align*}
    \left( \frac{1}{1+\varepsilon} + \frac{\varepsilon}{1+\varepsilon}x \right)^p \leq \frac{1}{1+\varepsilon} + \frac{\varepsilon}{1+\varepsilon}x^p.
\end{align*}
Multiplying with $(1+\varepsilon)^p$, 
\begin{align*}
    \left(1 + \varepsilon x \right)^p \leq (1+ \varepsilon)^{p-1} + \left(\frac{1+ \varepsilon}{\varepsilon}\right)^{p-1} (\varepsilon x)^p.
\end{align*}
With $x=\frac{b}{\varepsilon a}$, this yields the inequality stated above.
\item The function $f: [0,\infty) \to \R$, $f(x) := 1 +  p  x^p +px  - \left( 1+x \right)^p$ is increasing on $[0,\infty)$. Thus, $f(x) \geq f(0) = 0$, i,plying $
0 \leq 1 +  p  \left( \frac{a}{b} \right)^p +p \frac{a}{b}   - \left( \frac{a}{b} +1 \right)^p$.
\item Using the convexity of $ t \mapsto |t|^p$, Lindqvist has shown in \cite[p. 74]{lindqvist2017notes} that $|y|^p \geq |x|^p + p  |x|^{p-2} x ( y-x )$ for any $x,y \in \mathbb{R}^d$ and any $p \geq 1$. This inequality in one dimension with $y=-b$ and $x=a+b$ together with  Lemma \ref{lemma:basicineq}\eqref{eq:algineqcd1} with $\varepsilon=1$ yield the claimed inequality.
\end{enumerate}\end{proof}
\end{appendix}

\section*{Acknowledgments}
The work of the second author was partially supported by the National Science Foundation, USA, through the award DMS-2345500. 

\bibliographystyle{alpha}
\bibliography{sn-bibliography.bib}


%
%
%
\end{document}